\documentclass[12pt, reqno]{amsart} % AMS article style
\usepackage{amssymb,amscd,amsfonts,amsbsy}
\usepackage{latexsym}
\usepackage{exscale}
\usepackage{amsmath,amsthm,amsfonts}
\usepackage{mathrsfs}
\usepackage{xcolor} % Enables the creations of colors.
\usepackage{cite} 
\usepackage{esint} % For \fint symbol
\usepackage{amssymb} % For \lesssim symbol
\usepackage{stmaryrd}
\usepackage[colorlinks=true, linkcolor=blue, citecolor=blue, urlcolor=blue, backref=page]{hyperref}% Add coloured links and numeric back references

\renewcommand*{\backrefalt}[4]{%
  \ifcase #1\relax
  \or\ #2%
  \else\ #2%
  \fi
}

\calclayout

\allowdisplaybreaks[3] % Allow equation break between pages
\newtheorem{theorem}{Theorem}[section]
\newtheorem{proposition}[theorem]{Proposition}
\newtheorem{lemma}[theorem]{Lemma}

\newtheorem{definition}[theorem]{Definition}

\newtheorem{Remark}[theorem]{Remark}
\numberwithin{equation}{section}

\makeatletter
\@addtoreset{equation}{section}
\makeatother

\DeclareMathOperator{\supp}{supp}

\def\Z{\mathbb{Z}}

\def\R{\mathbb{R}}

\def\supp{\operatorname{supp}}

\begin{document}

	\title[Extrapolation to product Morrey-Herz spaces]
	{Extrapolation to product Morrey-Herz spaces and applications}

	\author[X. Cen]{Xi Cen}
	\address{Xi Cen\\
		School of Science\\
		China University of Mining and
		Technology-Beijing\\
		Beijing 100083 \\
		People's Republic of China}\email{xicenmath@gmail.com}

	\author[Z. Song]{Zichen Song}
	\address{Zichen Song\\
		School of Mathematics and Statistics\\
		Xinyang Normal University\\
		Xinyang 464000\\
		People's Republic of China}\email{zcsongmath@gmail.com}

	

	\thanks{Corresponding author: Xi Cen, Email: xicenmath@gmail.com}
	
	\begin{abstract} 
		
	The purpose of this paper is threefold. First, we introduce product Morrey-Herz spaces and product block-Herz spaces, establish their duality, and prove the boundedness of the strong maximal operator on product block-Herz spaces; these results provide the foundation for extrapolation. Second, using the Rubio de Francia iteration method, we establish extrapolation results on product Morrey-Herz spaces. Finally, we give applications to Fefferman-Stein vector-valued strong maximal inequalities, the John-Nirenberg inequality, a characterization of little bmo in terms of product Morrey-Herz spaces, and the boundedness of bi-parameter Calderón-Zygmund operators and their commutators.
	\end{abstract}

	\maketitle
	\section{Introduction}\label{sec1}
	In this paper, we investigate extrapolation theory on product Morrey-Herz spaces and its applications.
	The foundational work in classical extrapolation theory was pioneered by Fefferman and Francia \cite{Fra.1982, Fra.1984}. 
	More recently, this theory has been extended to Lebesgue spaces with variable exponents \cite{Cruz2006}, rearrangement-invariant Banach function spaces \cite{Cruz2004}, Banach function spaces \cite{Cruz2011}, variable Herz spaces \cite{Ho2020_1, Ho2020_2, Ho2021_1}, and related settings.
	
	The classical Morrey space was introduced by Morrey in 1938 to study the regularity of elliptic partial differential equations \cite{Morrey1938} and is an important and natural generalization of Lebesgue spaces.
	Many researchers have studied the theory and applications of Morrey spaces; for example, see \cite{Gul.2018, Liu2014, Naka2018}. The theory of extrapolation has also been extended to classical Morrey spaces. Notable contributions in this area include works by Duoandikoetxea et al. \cite{Duo2018, Duo2021}, Kokilashvili et al. \cite{Koki2019}, Ho \cite{Ho2017_1}, and Rosenthal et al. \cite{Rose2016}.
	
	In addition to the aforementioned function spaces, the classical Herz space \( K_{p,q}^{\alpha}(\mathbb{R}^n) \) is also of interest. Introduced by Herz in 1968 \cite{Herz1968} to study Fourier series and Fourier transform, classical Herz spaces serve as extensions of Lebesgue spaces. Numerous significant properties of Lebesgue spaces, such as the boundedness of singular integral operators, the mapping properties of fractional integral operators, complex interpolation, and Hausdorff-Young inequalities, have been extended to Herz spaces; refer to \cite{Fei2008, Ho2019_1, Lu1996}. In works like \cite{Ra2009, Ra2012}, solutions of partial differential equations are also investigated using Herz spaces \( K_{p,q}^{\alpha}(\mathbb{R}^n) \). For detailed information regarding classical Herz spaces \( K_{p,q}^{\alpha}(\mathbb{R}^n) \), readers may refer to \cite{Lu2008}.
	
	To amalgamate Morrey spaces and Herz spaces, classical Morrey-Herz spaces \( MK_{p,q}^{\alpha,\lambda}(\mathbb{R}^n) \) were introduced and investigated in \cite{Lu2005}. Subsequently, in works such as \cite{Ab.2021, Shi2009}, Morrey-Herz spaces were demonstrated to be appropriate generalizations of both Morrey spaces and Herz spaces, facilitating the study of the boundedness of integral operators in harmonic analysis. Moreover, Morrey-Herz spaces have recently been extended to the variable setting, as evidenced by works like \cite{Dri2016, Izuki2010}.
	
	In 1982, Fefferman and Stein \cite{Fe.1982} introduced the bi-parameter singular integral operator and established its boundedness on weighted product Lebesgue spaces. This result initiated the study of bi-parameter operators and bi-parameter function spaces in harmonic analysis, see \cite{Chen2020,Fe.1987,Hong2018} for more results on the bi-parameter theory. 
	
	As is well-known, bi-parameter operators exhibit higher singularity compared to their one-parameter counterparts. Extending the one-parameter theory in harmonic analysis to its bi-parameter version is both significant and challenging. It is worth noting that Ho recently developed bi-parameter extrapolation theory in \cite{Ho2016_1}, while Kokilashvili established the boundedness of bi-parameter operators in \cite{Koki.2018}. However, the extrapolation theory on product Morrey-Herz spaces remains unknown. Addressing this issue constitutes the main focus of this paper.
	
	Establishing extrapolation theory is challenging for spaces whose norms are not generated by Banach function spaces, in particular for Morrey-type spaces \cite[Sect.5]{P.A.2009}. Since Morrey-type spaces are not Banach function spaces in general \cite{Sawano2015} and product Morrey-Herz spaces are inherently bi-parameter spaces, additional difficulties arise beyond those in the classical setting. To develop extrapolation theory on product Morrey-Herz spaces, we use product block-Herz spaces as pre-dual spaces. A key point is that the strong maximal operator is bounded on product block-Herz spaces under suitable assumptions, which allows us to establish the desired extrapolation results.
	
	The extrapolation theory on product Morrey-Herz spaces has several applications. For instance, it enables us to prove boundedness results for bi-parameter operators once the corresponding weighted product Lebesgue estimates are available. It also allows us to formulate the John-Nirenberg inequality in the product Morrey-Herz norm, which yields a new characterization of little bmo in terms of product Morrey-Herz spaces. These applications are obtained by combining our extrapolation results with known inequalities and characterizations from \cite{Ho2009, Ho2011, Ho2016_1, Izuki2012}.
	
	The paper is organized as follows. In Section~\ref{sec2}, we recall the necessary definitions and preliminary lemmas. In Section~\ref{sec3}, we establish the duality between product Morrey-Herz spaces and product block-Herz spaces and prove the boundedness of the strong maximal operator \(M_S\) on product block-Herz spaces. In Section~\ref{sec4}, we prove the extrapolation theorem on product Morrey-Herz spaces. Finally, in Section~\ref{sec5}, we apply the extrapolation theorem to Fefferman-Stein vector-valued strong maximal inequalities, the John-Nirenberg inequality, a new characterization of little bmo, and the boundedness of bi-parameter Calderón-Zygmund operators and their commutators.
	
	\section {Preliminaries}\label{sec2}
	Throughout the paper, the following notations are employed.
	\begin{itemize}
		\item  \( \mathcal{M}(\mathbb{R}^d) \) and \( L_{{loc}}(\mathbb{R}^d) \) denote the class of Lebesgue measurable functions and the class of locally integrable functions on \( \mathbb{R}^d \), respectively.
		\item For \( 1 \leq p \leq \infty \), \( p^{\prime} \) denotes the conjugate exponent of \( p \), such that \( \frac{1}{p} + \frac{1}{p^{\prime}} = 1 \).
		\item For any \( r > 0 \) and \( x \in \mathbb{R}^d \), \( Q(x, r) \) represents the cube in \( \mathbb{R}^d \) centered at \( x \) with side length \( r \). The set of all such cubes is denoted by \( \mathcal{Q} = \{ Q(x, r) : x \in \mathbb{R}^d, r > 0 \} \).
		\item For any \( r, s > 0 \) and \( z = (x, y) \in \mathbb{R}^n \times \mathbb{R}^m \), \( R(z, r, s) = Q(x, r) \times Q(y, s) \), and \( \mathcal{R} \) denotes the set of all \( R(z, r, s) \) for any \( r, s > 0 \) and \( z \in \mathbb{R}^n \times \mathbb{R}^m \).
	\item The dyadic cubes and rectangles centered at the origin are denoted by \(Q_k=Q(0,2^k)\) and \(R_{i,j}=R(0,2^i,2^j)\), where \(k,i,j\in\mathbb Z\). We write \(I_i\times J_j\) for the dyadic product annulus in \(\mathbb R^n\times\mathbb R^m\), where \(I_i=Q_i\setminus Q_{i-1}\) and \(J_j=Q_j\setminus Q_{j-1}\).
		\item \( \chi_E \) denotes the characteristic function of a measurable set \( E \), and \( |E| \) represents the Lebesgue measure of \( E \). \( \chi_{i,j} \) signifies the characteristic function on \( I_i \times J_j \).
	\end{itemize}

	The symbols \( \mathbb{N} \) and \( \mathbb{C} \) denote the collections of all positive integers and all complex numbers, respectively. The collection of all non-negative integers is represented by \( \mathbb{N}_0 \). The notation \( A \lesssim B \) signifies that \( A \leq C B \) for some constant \( C > 0 \), while \( A \approx B \) indicates that \( A \lesssim B \) and \( B \lesssim A \).
	To denote that the implicit constant \( C \) actually depends on certain parameters such as \( \alpha, \beta, \gamma \), etc., we append these parameters to the symbols. For instance, \( A \lesssim_{\alpha, \beta} B \) implies that the implicit constant depends on the parameters \( \alpha \) and \( \beta \).
	
	For \( \alpha \in \mathbb{R} \) and \( 0 < p, q \leq \infty \), the classical product Herz space \( \vec{\dot{K}}_{p,q}^\alpha(\mathbb{R}^n \times \mathbb{R}^m) \), another extension of product Lebesgue spaces, comprises all \( f \in L_{\text{loc}}^p(\mathbb{R}^n \times \mathbb{R}^m \setminus \{(0,0)\}) \) such that \( \|f\|_{\vec{\dot{K}}_{p,q}^\alpha(\mathbb{R}^n \times \mathbb{R}^m)} < \infty \), where
	\[
	\|f\|_{\vec{\dot{K}}_{p,q}^\alpha(\mathbb{R}^n \times \mathbb{R}^m)} = \left\{ \sum_{i, j \in \mathbb{Z}} 2^{(i+j) q\alpha} \|f \chi_{i,j}\|_{L^p(\mathbb{R}^n \times \mathbb{R}^m)}^q \right\}^{\frac{1}{q}}
	\]
	This norm is modified appropriately when \(p = \infty \) or \(q = \infty \). The following proposition records a simple property of the product Herz space, already proven in \cite{Lu1995, Lu2008}.
	\begin{proposition}
		The space ${\vec{\dot{K}}_{p, q}^{\alpha}(\mathbb{R}^n \times \mathbb{R}^m)}$ is a quasi-Banach space, and if $p, q \geq 1$, then ${\vec{\dot{K}}_{p, q}^{\alpha}(\mathbb{R}^n \times \mathbb{R}^m)}$ is a Banach space.
	\end{proposition}
	As demonstrated in previous studies \cite{Ho2016_1,Koki.2018}, extrapolation theory in product spaces depends heavily on strong $A_p^*$ weights. Below we give the definition of $A_p^*$.
	\begin{definition}
		For $1<p<\infty$, a non-negative locally integrable function $\omega$ on $\mathbb{R}^n \times \mathbb{R}^m$ is said to be in $A_p^*$ if
		$$
		{\left[ \omega  \right]_{A_p^{\rm{*}}}}:=\mathop {\sup }\limits_{R \in \mathcal R} \left(\frac{1}{|R|}\int_R \omega\right)\left(\frac{1}{|R|}\int_R \omega^{1-p'}\right)^{p-1}<\infty .
		$$
		A non-negative locally integrable function $\omega$ is said to be an $A_1^*$ weight if
		$$
		{\left[ \omega  \right]_{A_1^{\rm{*}}}}:=\mathop {\sup }\limits_{R \in \mathcal R} \left(\frac{1}{|R|}\int_R \omega\right){\left\| {{\omega ^{ - 1}}} \right\|_{{L^\infty }(R)}} < \infty.
		$$
		We denote by $A_{\infty}^*$ the union of all $A_p^*(1 \leq p<\infty)$ functions.
	\end{definition}
	It is important to note that the product Herz space discussed here differs slightly from the classical Herz space. In general, rectangles are used instead of cubes (or balls) in classical Herz spaces. As we will observe, the product Herz space is more suitable for studying certain bi-parameter problems compared to classical Herz-type spaces.
	
	We introduce the strong maximal operator, which is as important as the classical Hardy-Littlewood maximal operator and serves as a fundamental tool for the study of bi-parameter problems in harmonic analysis, especially in the theory of bi-parameter extrapolation \cite{Koki.2018}. Before delving into its applications, we will briefly review the definition of the strong maximal operator. Bi-parameter harmonic analysis deals with functions of two variables and extends the concepts and techniques of classical harmonic analysis to this setting. The strong maximal operator plays a crucial role in this paper.
	
	For $f \in L_{ {loc}}\left(\mathbb{R}^n \times \mathbb{R}^m\right)$, the strong maximal operator $M_S$ is defined by
	$$
	M_S f(z)=\sup _{z \in R \in \mathcal{R}} \frac{1}{|R|} \int_R\left|f\left(z^{\prime}\right)\right| d z^{\prime} .
	$$
	Recall the space $L^p\left(\mathbb{R}^n \times \mathbb{R}^m, \omega \right)$ consists of all $f \in \mathcal{M}\left(\mathbb{R}^n \times \mathbb{R}^m\right)$ such that
	$$
	\|f\|_{L^p\left(\mathbb{R}^n \times \mathbb{R}^m,\omega \right)}=\left(\int_{\mathbb{R}^n \times \mathbb{R}^m}|f(z)|^p \omega (z) d z\right)^{\frac{1}{p}}<\infty.
	$$
	
	The following lemma, proved in \cite{Gar2011}, states that $M_S$ is bounded on product weighted Lebesgue spaces.
	\begin{lemma}
		Let $1<p<\infty$. Then $M_S$ is bounded on $L^p\left(\mathbb{R}^n \times \mathbb{R}^m,\omega \right)$
		if and only if $\omega \in A_p^*\left(\mathbb{R}^n \times \mathbb{R}^m\right)$.
	\end{lemma}
	Additionally, it is worth noting that many other bi-parameter operators are also bounded in weighted Lebesgue spaces \cite{Fe.1987,Hong2018}. The next lemmas regarding the dual space for product Herz spaces and the boundedness of the strong maximal operator $M_S$ on $\vec{\dot{K}}_{p,q}^\alpha\left(\mathbb{R}^n \times \mathbb{R}^m\right)$ were proved in \cite{Wei2021_1}.
	\begin{lemma}\label{Hölder-P-H}
		Let $\alpha \in \mathbb{R}, 0<q<\infty, 1 \leq p<\infty$, and $\frac{1}{q}+\frac{1}{q^{\prime}}=1$, where $q^{\prime}=\infty$ if $0<q\leq 1$. Then
		\begin{align*}
			\left({\vec{\dot{K}}_{p, q}^{\alpha}}\left( {{\R^n} \times {\R^m}} \right)\right)^* & ={\vec{\dot{K}}_{p^{\prime}, q^{\prime}}^{-\alpha}}\left( {{\R^n} \times {\R^m}} \right).
		\end{align*}
		Moreover, we have
		$$
		{\left\| f \right\|_{\vec {\dot K}_{p',q'}^{-\alpha}\left(\mathbb{R}^n \times \mathbb{R}^m\right)}} \approx 
		{\left\| {{L_f}} \right\|_{{{\left( {\vec {\dot K}_{p,q}^{\alpha}\left(\mathbb{R}^n \times \mathbb{R}^m\right)} \right)}^*}}} 
		:=
		\mathop {\sup }\limits_{{{\left\|g\right\|}_{\vec {\dot K}_{{p},{q}}^{\alpha }\left(\mathbb{R}^n \times \mathbb{R}^m\right)}} \le 1} {\left\| {fg} \right\|_{{L^1\left(\mathbb{R}^n \times \mathbb{R}^m\right)}}}.
		$$
	\end{lemma}
	\begin{lemma}\label{MSHerz}
		Let $0<q<\infty$, $ 1<p<\infty$, and
		$$
		\max \left\{-\frac{n}{p},-\frac{m}{p}\right\}<\alpha<\min \left\{n\left(1-\frac{1}{p}\right), m\left(1-\frac{1}{p}\right)\right\} .
		$$
		Then $M_S$ is bounded on $\vec{\dot{K}}_{p,q}^\alpha\left( {{\R^n} \times {\R^m}} \right)$.
	\end{lemma}
	We now define the product Morrey-Herz spaces, which serve as generalizations of classical Morrey-Herz spaces.
	\begin{definition}
		For $\alpha \in \mathbb{R}, 0<p, q \leq \infty$, and $0 \leq \lambda<\infty$, the product Morrey-Herz spaces $M \vec{\dot{K}}_{p, q}^{\alpha, \lambda}\left( {{\R^n} \times {\R^m}} \right)$ are defined by
		$$
		M \vec{\dot{K}}_{p, q}^{\alpha, \lambda}\left(\mathbb{R}^n \times \R^m\right):=\left\{f \in L_{{loc}}^p\left(\mathbb{R}^n \times \R^m \backslash\{(0,0)\}\right):\|f\|_{M \vec{\dot{K}}_{p, q}^{\alpha, \lambda}\left(\mathbb{R}^n \times \R^m\right)}<\infty\right\}
		$$
		where
		\begin{align*}
			\|f\|_{M \vec{\dot{K}}_{p, q}^{\alpha, \lambda}\left( {{\R^n} \times {\R^m}} \right)} :=& \sup_{L_1,L_2 \in \mathbb{Z}}2^{-(L_1+L_2)\lambda}\left\|\left\{ 2^{(i+j)\alpha}\left\|{f{\chi _{i,j}}}\right\|_{L^p\left(\mathbb{R}^n \times \mathbb{R}^m\right)}\right\}_{i\le L_1,j \leq L_2}\right\|_{l^q}\\
			=& \mathop {\sup }\limits_{L_1,L_2 \in \mathbb{Z}} {2^{ - (L_1+L_2)\lambda }}{\left\| {f{\chi _{{R_{L_1,L_2}}}}} \right\|_{\vec{\dot{K}}_{p, q}^{\alpha}\left( {{\R^n} \times {\R^m}} \right)}}.    
		\end{align*}
	\end{definition}
	It is noteworthy that ${M \vec{\dot{K}}_{p, q}^{\alpha, \lambda}}\left( {{\R^n} \times {\R^m}} \right)={ \vec{\dot{K}}_{p, q}^{\alpha}}\left( {{\R^n} \times {\R^m}} \right)$ when $\lambda = 0$. The following lemma establishes that if $\alpha \in \mathbb{R}$, $ 0<p, q<\infty$, and $0<\lambda<\min\{\alpha+n/p,\alpha+m/p\}$, then the characteristic functions $\chi_{R_{l_1,l_2}} \in {M \vec{\dot{K}}_{p, q}^{\alpha, \lambda}}\left( {{\mathbb{R}^n} \times {\mathbb{R}^m}} \right)$ for every $l_1,l_2 \in \mathbb{Z}$.
	
	\begin{lemma}\label{chara.PMK}
		Let $\alpha \in \mathbb{R}$, $ 0<p, q<\infty$ and $0<\lambda<\min\{\alpha+n/p,\alpha+m/p\}$. Then for every $l_1,l_2 \in \mathbb{Z}$, we have $\chi_{R_{l_1,l_2}} \in {M \vec{\dot{K}}_{p, q}^{\alpha, \lambda}}\left( {{\R^n} \times {\R^m}} \right)$.
	\end{lemma}
	\begin{proof}
		For all $L_1,L_2 \in \mathbb{Z}$, we have 
		\begin{align*}
			&\quad
			2^{-(L_1+L_2)\lambda}
			\left(\sum_{i\le L_1,j\le L_2} 
			2^{({i+j}) q \alpha}\left\|\chi_{R_{l_1,l_2}} \chi_{I_i \times J_j}\right\|_{L^p}^q\right)^{1 / q} 
			\\&=
			2^{-(L_1+L_2)\lambda}
			\left(\sum_{i\le \min\{L_1,l_1\},j\le \min\{L_2,l_2\}} 
			2^{({i+j}) q \alpha}\left\|\chi_{I_i \times J_j}\right\|_{L^p}^q\right)^{1 / q} 
			\\&=
			2^{-(L_1+L_2)\lambda}
			\left(\sum_{i\le \min\{L_1,l_1\},j\le \min\{L_2,l_2\}} 
			2^{({i+j}) q \alpha}
			(1-2^{-n})^{q/p}
			(1-2^{-m})^{q/p}
			2^{inq/p}
			2^{jmq/p}
			\right)^{1 / q} 
			\\&\approx
			2^{-(L_1+L_2)\lambda}
			\left(\sum_{i\le \min\{L_1,l_1\},j\le \min\{L_2,l_2\}} 
			2^{iq(\alpha+n/p)}
			2^{jq(\alpha+m/p)}
			\right)^{1 / q} 
			\\&=
			2^{-L_1\lambda}
			\left(
			\sum_{i\le \min\{L_1,l_1\}}
			2^{iq(\alpha+n/p)}
			\right)^{1 / q} 
			\cdot
			2^{-L_2\lambda}
			\left(
			\sum_{j\le \min\{L_2,l_2\}} 
			2^{jq(\alpha+m/p)}
			\right)^{1 / q} 
			\\&\approx
			2^{-L_1\lambda}
			2^{\min\{L_1,l_1\}(\alpha+n/p)}
			\cdot
			2^{-L_2\lambda}
			2^{\min\{L_2,l_2\}(\alpha+m/p)}
		\end{align*}
		Consequently, since $0<\lambda<\min\{\alpha+n/p,\alpha+m/p\}$,
		\begin{align}\label{ineq.1}
			\|\chi_{R_{l_1,l_2}}\|_{{M \vec{\dot{K}}_{p, q}^{\alpha, \lambda}}}
			&=\notag
			\sup_{L_1,L_2\in\mathbb{Z}}
			\left\{
			2^{-L_1\lambda}
			2^{\min\{L_1,l_1\}(\alpha+n/p)}
			\cdot
			2^{-L_2\lambda}
			2^{\min\{L_2,l_2\}(\alpha+m/p)}
			\right\}
			\\&\approx
			2^{l_1(\alpha+n/p-\lambda)}
			2^{l_2(\alpha+m/p-\lambda)},
		\end{align}
		where the implicit constants are independent of $l_1$ and $l_2$.
		This proves the assertion.
	\end{proof}
	
	\section{Duality for product Morrey-Herz spaces}\label{sec3}
	As is well-known, many function spaces have pre-duals that manifest as block-type spaces, such as Herz spaces \cite{Yang1999}, Morrey Banach spaces \cite{Ho2017_1}, mixed Morrey spaces \cite{Nogayama2019_1}, generalized Morrey spaces \cite{Ho2017_2}, among others. In this section, we aim to establish more general block-type spaces termed product block-Herz spaces, which serve as the pre-duals of product Morrey-Herz spaces. Additionally, we will demonstrate the boundedness of the strong maximal operator on product block-Herz spaces.
	\begin{definition}\label{defBK}
		For $\alpha \in \mathbb{R}$, $0<p, q \leq \infty$, and $0<\lambda<\infty$, a function $b \in \mathcal{M}\left(\mathbb{R}^n \times \mathbb{R}^m\right)$ is said to be a product Herz-block, and we write $b \in {b \vec{\dot{K}}_{p, q}^{\alpha, \lambda}}\left( {{\R^n} \times {\R^m}} \right)$ if it is supported in a {product rectangle} $R _ {l_1,l_2}$ where $l_1,l_2 \in \mathbb{Z}$, and
		$$
		\|b\|_{{ \vec{\dot{K}}_{p, q}^{\alpha}}\left(\mathbb{R}^n \times \mathbb{R}^m\right)} \le {2^{-\lambda(l_1+l_2)}},
		$$
		where $b$ is called a product Herz-block. The product block-Herz space ${B \vec{\dot{K}}_{p, q}^{\alpha, \lambda}}\left(\mathbb{R}^n \times \mathbb{R}^m\right)$ is defined as
		$$
		{B \vec{\dot{K}}_{p, q}^{\alpha, \lambda}}\left(\mathbb{R}^n \times \mathbb{R}^m\right)=\left\{\sum_{k=1}^{\infty} \lambda_k b_k: \sum_{k=1}^{\infty}\left|\lambda_k\right|<\infty, b _k \in {b \vec{\dot{K}}_{p, q}^{\alpha, \lambda}}\left(\mathbb{R}^n \times \mathbb{R}^m\right)\right\} .
		$$
		The space ${B \vec{\dot{K}}_{p, q}^{\alpha, \lambda}}\left(\mathbb{R}^n \times \mathbb{R}^m\right)$ is endowed with the norm
		$$
		\begin{aligned}
			\|f\|_{B \vec{\dot{K}}_{p, q}^{\alpha, \lambda}\left(\mathbb{R}^n \times \mathbb{R}^m\right)} 
			=\inf \left\{\sum_{k=1}^{\infty}\left|\lambda_k\right|: f=\sum_{k=1}^{\infty} \lambda_k b_k, b _k \in {b \vec{\dot{K}}_{p, q}^{\alpha, \lambda}}\left(\mathbb{R}^n \times \mathbb{R}^m\right)\right\} .
		\end{aligned}
		$$
	\end{definition}
	Definition~\ref{defBK} extends the classical block-Herz spaces in \cite{Wei2023_1} to the bi-parameter setting. Other generalizations of classical block spaces include mixed block spaces \cite{Nogayama2019_1,Nogayama2021_1} and block spaces associated with Banach function spaces \cite{Ho2017_1}. The Fatou property of the classical one-parameter block spaces was proved in \cite{Sawano2015}.
	
	In this section, we present our main result, which reveals that the pre-dual spaces of product Morrey-Herz spaces can be considered as the product block-Herz spaces defined earlier.
	\begin{theorem}\label{PHMdual}
		Let $\alpha \in \mathbb{R}$, $ 1<p, q<\infty$, and $0<\lambda<\infty$. Then 
		$$
		\left({B\vec{\dot{K}}_{p, q}^{\alpha, \lambda}}\left(\mathbb{R}^n \times \R^m\right)\right)^*=M \vec{\dot{K}}_{p^{\prime}, q^{\prime}}^{-\alpha, \lambda}\left(\mathbb{R}^n \times \R^m\right)
		$$
		Furthermore, we have
		$${\left\| f \right\|_{M\vec {\dot K}_{{p^\prime },{q^\prime }}^{ - \alpha ,\lambda }\left(\mathbb{R}^n \times \mathbb{R}^m\right)}} \approx {\left\| {{L_f}} \right\|_{{{\left( {B\vec {\dot K}_{p,q}^{\alpha ,\lambda }\left(\mathbb{R}^n \times \mathbb{R}^m\right)} \right)}^*}}}:=\mathop {\sup }\limits_{{{\left\|g \right\|}_{B\vec {\dot K}_{p,q}^{\alpha ,\lambda }\left(\mathbb{R}^n \times \mathbb{R}^m\right)}} \le 1} {\left\| {fg} \right\|_{{L^1\left(\mathbb{R}^n \times \mathbb{R}^m\right)}}}$$
	\end{theorem}
	\begin{proof}
		Let $b \in {b \vec{\dot{K}}_{p, q}^{\alpha, \lambda}}$  supported in $R_{l_1,l_2}$ for some $l_1,l_2 \in \mathbb{Z}$. For any $f \in {M \vec{\dot{K}}_{p', q'}^{-\alpha, \lambda}}$, from Lemma \ref{Hölder-P-H}, we get
		$$
		\begin{aligned}
			\int_{\mathbb{R}^n \times \R^m}|f(x) b(x)| d x & \lesssim\left\|f \chi_{R_{l_1,l_2}}\right\|_{\vec{\dot{K}}_{p', q'}^{-\alpha}}\left\|b \chi_{R_{l_1,l_2}}\right\|_{\vec{\dot{K}}_{p, q}^{\alpha}} \\
			& \lesssim {2^{-\lambda(l_1+l_2)}}\left\|f \chi_{R_{l_1,l_2}}\right\|_{\vec{\dot{K}}_{p^{\prime}, q^{\prime}}^{-\alpha}}\\
			&\le \|f\|_{{M \vec{\dot{K}}_{p', q'}^{-\alpha, \lambda}}}.
		\end{aligned}
		$$
		Therefore, for any $g \in {B \vec{\dot{K}}_{p, q}^{\alpha, \lambda}}$, we can write $$g=\sum_{k=1}^{\infty}\lambda_k b_k,$$ where $\sum_{k=1}^{\infty}|\lambda_k| \leq 2\|g\|_{B \vec{\dot{K}}_{p, q}^{\alpha, \lambda}}$. By the definition of ${M \vec{\dot{K}}_{p', q'}^{-\alpha, \lambda}}$, we get
		$$
		\int_{\mathbb{R}^n \times \R^m}|f(x) g(x)| d x \lesssim \sum_{k=1}^{\infty}\left|\lambda_k\right| \int_{\mathbb{R}^n \times \R^m}\left|f(x) b_k(x)\right| d x \lesssim\|g\|_{{B \vec{\dot{K}}_{p, q}^{\alpha, \lambda}}}\|f\|_{{M \vec{\dot{K}}_{p', q'}^{-\alpha, \lambda}}}.
		$$
		Thus, ${M \vec{\dot{K}}_{p', q'}^{-\alpha, \lambda}} \hookrightarrow \left({B \vec{\dot{K}}_{p, q}^{\alpha, \lambda}}\right)^*$.
		Then consider the reverse embedding.
		
		Next, we will show that any $T \in \left({B \vec{\dot{K}}_{p, q}^{\alpha, \lambda}}\right)^*$ can be identified with a function in $M \vec{\dot{K}}_{p^{\prime}, q^{\prime}}^{-\alpha, \lambda}$.
		For any $l_1,l_2 \in \mathbb{Z}$ and $g \in {\vec{\dot{K}}_{p, q}^{\alpha}}$, let
		$$
		G=\frac{g \chi_{R_{l_1,l_2}}}{2^{(l_1+l_2) \lambda}\left\|g \chi_{R_{l_1,l_2}}\right\|_{{\vec{\dot{K}}_{p, q}^{\alpha}}}}.
		$$
		Hence, $G \in {b\vec{\dot{K}}_{p, q}^{\alpha,\lambda}}$. From Definition \ref{defBK}, for any  $b \in {b \vec{\dot{K}}_{p, q}^{\alpha,\lambda}}$, we have $\|b\|_{B \vec{\dot{K}}_{p, q}^{\alpha, \lambda}} \leq 1$. Thus, $\|G\|_{B \vec{\dot{K}}_{p, q}^{\alpha, \lambda}} \leq 1$, and then
		\begin{align}\label{BKdual_1}
			\left\|g \chi_{R_{l_1,l_2}}\right\|_{B \vec{\dot{K}}_{p, q}^{\alpha, \lambda}} \leq 2^{(l_1+l_2) \lambda}\left\|g \chi_{R_{l_1,l_2}}\right\|_{{\vec{\dot{K}}_{p, q}^{\alpha}}}.
		\end{align}
		Using \eqref{BKdual_1} and the condition $T \in \left({B \vec{\dot{K}}_{p, q}^{\alpha, \lambda}}\right)^*$, we get
		\begin{align}\label{BKdual_2}
			\left|T\left(g \chi_{R_{l_1,l_2}}\right)\right| \lesssim\left\|g \chi_{R_{l_1,l_2}}\right\|_{B \vec{\dot{K}}_{p, q}^{\alpha, \lambda}} \le 2^{(l_1+l_2) \lambda}\left\|g \chi_{R_{l_1,l_2}}\right\|_{{\vec{\dot{K}}_{p, q}^{\alpha}}}
		\end{align}
		For $l_1,l_2 \in \mathbb{Z}$, define $X_{l_1,l_2}=\left\{g \chi_{R_{l_1,l_2}}: g \in {\vec{\dot{K}}_{p, q}^{\alpha}}\right\}$, then $X_{l_1,l_2} \subseteq {\vec{\dot{K}}_{p, q}^{\alpha}}$. For any $l_1,l_2 \in \mathbb{Z}$ and $T \in \left({B \vec{\dot{K}}_{p, q}^{\alpha, \lambda}}\right)^*$, we define a linear functional $L_{l_1,l_2}: X_{l_1,l_2} \rightarrow \mathbb{C}$ by
		$
		L_{l_1,l_2}:=T|_{X_{l_1,l_2}}.
		$
		By \eqref{BKdual_2}, $L_{l_1,l_2}$ is bounded on $X_{l_1,l_2}$. By the Hahn--Banach theorem, $L_{l_1,l_2}$ can be extended to a bounded linear functional on $\vec{\dot{K}}_{p,q}^{\alpha}$. 
		
		From Lemma \ref{Hölder-P-H}, there exists $f_{l_1,l_2} \in {\vec{\dot{K}}_{p^{\prime}, q^{\prime}}^{-\alpha}}$, such that for any $g \in {\vec{\dot{K}}_{p, q}^{\alpha}}$,
		$$
		L_{l_1,l_2}(g)=\int_{\mathbb{R}^n \times \R^m} f_{l_1,l_2}(x) g(x) d x.
		$$
		Without loss of generality, we can assume that $\supp f_{l_1,l_2} \subseteq R_{l_1,l_2}$.
		Note that for any $l_1,l_2, k_1,k_2 \in \mathbb{Z}$, and $R_{s_1,s_2}$ with $R_{s_1,s_2} \subseteq R_{l_1,l_2} \cap R_{k_1,k_2}$,
		$$
		\int_{R_{s_1,s_2}} f_{l_1,l_2}(x) d x=L\left(\chi_{R_{s_1,s_2}}\right)=\int_{R_{s_1,s_2}} f_{k_1,k_2}(x) d x
		$$
		Hence, $f_{l_1,l_2}=f_{k_1,k_2}$ almost everywhere on $R_{s_1,s_2} \subseteq R_{l_1,l_2} \cap R_{k_1,k_2}$. Therefore, there is a unique measurable function $f$ such that $f(x)=f_{l_1,l_2}(x)$ on $R_{l_1,l_2}$ for all $l_1,l_2\in \Z$.
		
		Next, we will show that $f \in {M \vec{\dot{K}}_{p^{\prime}, q^{\prime}}^{-\alpha, \lambda}}$.
		For any $h \in {\vec{\dot{K}}_{p, q}^{\alpha}}$ and $l_1,l_2 \in \mathbb{Z}$, we have
		$$
		H=\frac{h \chi_{R_{l_1,l_2}}}{2^{(l_1+l_2) \lambda}\left\|h \chi_{R_{l_1,l_2}}\right\|_{\vec{\dot{K}}_{p, q}^{\alpha}}}
		\in {b \vec{\dot{K}}_{p, q}^{\alpha, \lambda}}.$$ Thus $\|H\|_{{B \vec{\dot{K}}_{p, q}^{\alpha, \lambda}}} \leq 1$ as before, i.e.
		\begin{align}\label{BKdual_3}
			\left\|h \chi_{R_{l_1,l_2}}\right\|_{{B \vec{\dot{K}}_{p, q}^{\alpha, \lambda}}} \leq 2^{(l_1+l_2) \lambda}\left\|h \chi_{R_{l_1,l_2}}\right\|_{\vec{\dot{K}}_{p, q}^{\alpha}}
		\end{align}
		Since $H \in {b\vec{\dot{K}}_{p, q}^{\alpha,\lambda}}$ as in \eqref{BKdual_3}, by Lemma \ref{Hölder-P-H}, we get
		$$
		\begin{aligned}
			\frac{1}{2^{(l_1+l_2) \lambda}}\left\|f \chi_{R_{l_1,l_2}}\right\|_{\vec{\dot{K}}_{p^{\prime}, q^{\prime}}^{-\alpha}} & =\frac{1}{2^{(l_1+l_2) \lambda}} \sup _{\|h\|_{\vec{\dot{K}}_{p,q}^{\alpha}} \le 1}\left|\int_{R_{l_1,l_2}} f(x) h(x) d x\right| \\
			& =\sup _{\|h\|_{\vec{\dot{K}}_{p, q}^{\alpha}} \le 1}\left|\int_{R_{l_1,l_2}} f_{l_1,l_2}(x) \frac{h(x) \chi_{R_{l_1,l_2}}}{2^{(l_1+l_2) \lambda}} d x\right| \\
			& \lesssim\|T\|_{\left({B \vec{\dot{K}}_{p, q}^{\alpha, \lambda}}\right)^*} \sup _{\|h\|_{\vec{\dot{K}}_{p, q}^{\alpha}}\le1}\left\|\frac{h \chi_{R_{l_1,l_2}}}{2^{(l_1+l_2) \lambda}}\right\|_{{B \vec{\dot{K}}_{p, q}^{\alpha, \lambda}}} \lesssim\|T\|_{\left({B \vec{\dot{K}}_{p, q}^{\alpha, \lambda}}\right)^*}.
		\end{aligned}
		$$
		Hence, 
		$${\left\| f \right\|_{M\vec{ \dot K}_{p',q'}^{-\alpha ,\lambda }}} \lesssim {\left\| T \right\|_{{{\left( {B\vec {\dot K}_{p,q}^{\alpha ,\lambda }} \right)}^*}}}.$$
		As the functionals $T_f(g)=\int_{\mathbb{R}^n \times \R^m} f(x) g(x) d x$ and $T$ are identical on ${b \vec{\dot{K}}_{p, q}^{\alpha, \lambda}}$, and the set of finite linear combinations of ${b \vec{\dot{K}}_{p, q}^{\alpha, \lambda}}$ is dense in $B \vec{\dot{K}}_{p,q}^{\alpha, \lambda}$ by Definition \ref{defBK}, we conclude that $T_f=T$ and $\left({B \vec{\dot{K}}_{p, q}^{\alpha, \lambda}}\right)^* \hookrightarrow {M \vec{\dot{K}}_{p^{\prime}, q^{\prime}}^{-\alpha, \lambda}}$. 
	\end{proof}
	The duality between one-parameter Morrey-Herz-type spaces and block-type spaces has been established in works such as \cite{Sawano2011,Chiarenza1987}. Theorem~\ref{PHMdual} extends these duality results from the classical one-parameter setting to the bi-parameter setting. We next turn to the boundedness of bi-parameter operators on product Morrey-Herz spaces.
	
	The following proposition implies that $B \vec{\dot{K}}_{p,q}^{\alpha, \lambda}\left( {{\R^n} \times {\R^m}} \right)$ forms a Banach lattice. This result is crucial for proving the boundedness of the strong maximal operator $M_S$ on $B \vec{\dot{K}}_{p,q}^{\alpha, \lambda}\left( {{\R^n} \times {\R^m}} \right)$.
	\begin{proposition}\label{pre.extra.}
		Let $\alpha \in \mathbb{R}, 1<p, q<\infty, 0<\lambda<\infty$ and $f \in B \vec{\dot{K}}_{p,q}^{\alpha, \lambda}\left( {{\R^n} \times {\R^m}} \right)$. If $|g| \leq|f|$, then $\|g\|_{B \vec{\dot{K}}_{p,q}^{\alpha, \lambda}\left(\mathbb{R}^n \times \mathbb{R}^m\right)} \leq\|f\|_{B \vec{\dot{K}}_{p,q}^{\alpha, \lambda}\left(\mathbb{R}^n \times \mathbb{R}^m\right)}$.
	\end{proposition}
	\begin{proof}
		For any $f \in B \vec{\dot{K}}_{p,q}^{\alpha, \lambda}$ and $\epsilon>0$, there exists a family of ${b \vec{\dot{K}}_{p,q}^{\alpha, \lambda}}$ $\left\{b_i\right\}_{i=1}^{\infty}$ and a family of scalars $\left\{\lambda_i\right\}_{i=1}^{\infty}$, such that $f=$ $\sum_{i=1}^{\infty} \lambda_i b_i$ and
		$$
		\sum_{i=1}^{\infty}\left|\lambda_i\right| \leq(1+\epsilon)\|f\|_{B \vec{\dot{K}}_{p,q}^{\alpha, \lambda}} .
		$$
		Define $g$ by $g=\sum_{i=1}^{\infty} \lambda_i c_i$, where
		$$
		c_i(x)= \begin{cases}\frac{g(x)}{f(x)} b_i(x), & f(x) \neq 0 \\ 0, & f(x)=0 .\end{cases}
		$$
		By $|g| \leq|f|$, we have $c_i \in {b \vec{\dot{K}}_{p,q}^{\alpha, \lambda}}$ for every $i\in \mathbb{N}$.
		Therefore, $g \in B \vec{\dot{K}}_{p,q}^{\alpha, \lambda}$. Since $\epsilon>0$ is arbitrary, we also establish $\|g\|_{B \vec{\dot{K}}_{p,q}^{\alpha, \lambda}} \leq\|f\|_{B \vec{\dot{K}}_{p,q}^{\alpha, \lambda}}$.
	\end{proof}
	
	Using Proposition \ref{pre.extra.}, we have the following theorem.
	\begin{theorem}\label{BK.Ban}
		Let $\alpha \in \mathbb{R}, 1<p, q<\infty$, and $0<\lambda<\min\{-\alpha+n/p',-\alpha+m/p'\}$. 
		Then ${B \vec{\dot{K}}_{p,q}^{\alpha, \lambda}}\left( {{\R^n} \times {\R^m}} \right) \hookrightarrow L_{{loc }}^1\left( {{\R^n} \times {\R^m}} \right)$ and ${B \vec{\dot{K}}_{p,q}^{\alpha, \lambda}}\left( {{\R^n} \times {\R^m}} \right)$ is a Banach lattice.
	\end{theorem}
	\begin{proof}
		By Lemma \ref{chara.PMK}, $\chi_{R_{l_1,l_2}} \in {M \vec{\dot{K}}_{p^{\prime}, q^{\prime}}^{-\alpha, \lambda}}
		$ for any $l_1, l_2 \in \mathbb{Z}$, and then Theorem \ref{PHMdual} assures that $\chi_{R_{l_1,l_2}} \in\left({B \vec{\dot{K}}_{p, q}^{\alpha, \lambda}}\right)^*$. Therefore, for any $f \in {B \vec{\dot{K}}_{p, q}^{\alpha, \lambda}}$, Hölder's inequality yields
		\begin{align}\label{BKBan_1}
			\int_{R_{l_1,l_2}}|f(x)| d x \lesssim\left\|\chi_{R_{l_1,l_2}}\right\|_{{M \vec{\dot{K}}_{p^{\prime}, q^{\prime}}^{-\alpha, \lambda}}
			}\|f\|_{{B \vec{\dot{K}}_{p, q}^{\alpha, \lambda}}}
		\end{align}
		Hence, ${B \vec{\dot{K}}_{p, q}^{\alpha, \lambda}} \hookrightarrow L_{\mathrm{loc}}^1$.
		
		Next we prove ${B \vec{\dot{K}}_{p,q}^{\alpha, \lambda}}$ is a Banach space.
		Let $f_i \in {B \vec{\dot{K}}_{p, q}^{\alpha, \lambda}}$, where $i \in \mathbb{N}$, satisfy
		$$
		\sum_{i=1}^{\infty}\left\|f_i\right\|_{{B \vec{\dot{K}}_{p, q}^{\alpha, \lambda}}}<\infty .
		$$
		From \eqref{BKBan_1}, for any $l_1, l_2 \in \mathbb{Z}$, we have
		$$
		\int_{R_{l_1, l_2}} \sum_{i=1}^{\infty}\left|f_i(x)\right| d x \lesssim\left\|\chi_{R_{l_1, l_2}}\right\|_{M\vec{\dot{K}}_{p^{\prime}, q^{\prime}}^{-\alpha, \lambda}}\left(\sum_{i=1}^{\infty}\left\|f_i\right\|_{{B \vec{\dot{K}}_{p, q}^{\alpha, \lambda}}}\right) .
		$$
		Hence, $f=\sum_{i=1}^{\infty} f_i$ is a well defined Lebesgue measurable function and $f \in$ $L_{{loc }}^1\left(\mathbb{R}^n \times \R^m\right)$. 
		For any $\epsilon>0$, there exists a positive integer $N$ such that for any $n>N$, there holds
		\begin{align}\label{BKBan_2}
			\sum_{i=n}^{\infty}\left\|f_i\right\|_{{B \vec{\dot{K}}_{p, q}^{\alpha, \lambda}}}<\epsilon
		\end{align}
		By the definition of ${B \vec{\dot{K}}_{p, q}^{\alpha, \lambda}}$, for any $\epsilon>0$,
		$$
		f_i=\sum_{k=1}^{\infty} \lambda_{k, i} f_{k, i}
		$$
		where $f_{k, i} \in {b \vec{\dot{K}}_{p, q}^{\alpha, \lambda}}$, $ i, k \in \mathbb{N}$ and
		$$
		\sum_{k=1}^{\infty}\left|\lambda_{k, i}\right| \leq(1+\epsilon)\left\|f_i\right\|_{{B \vec{\dot{K}}_{p, q}^{\alpha, \lambda}}} .
		$$
		Then for any $1 \leq i \leq N$, there exists a $N_i \in \mathbb{N}$ such that
		\begin{align}\label{BKBan_3}
			\left\|f_i-\sum_{k=1}^{N_i} \lambda_{k, i} f_{k, i}\right\|_{{B \vec{\dot{K}}_{p, q}^{\alpha, \lambda}}} \leq \sum_{k=N_i+1}^{\infty}\left|\lambda_{k, i}\right|<2^{-i} \epsilon .
		\end{align}
		Therefore, for any $l_1,l_2 \in \mathbb{Z}$,
		$$
		\begin{aligned}
			& \int_{R_{l_1,l_2}}\left|f(x)-\sum_{i=1}^N \sum_{k=1}^{N_i} \lambda_{k, i} f_{k, i}(x)\right| d x \\
			& \quad \leq \int_{R_{l_1,l_2}}\left|f(x)-\sum_{i=1}^N f_i(x)\right| d x+\int_{R_{l_1,l_2}}\left|\sum_{i=1}^N f_i(x)-\sum_{i=1}^N \sum_{k=1}^{N_i} \lambda_{k, i} f_{k, i}(x)\right| d x\\
			& \quad \leq \int_{R_{l_1,l_2}} \sum_{i=N+1}^{\infty}\left|f_i(x)\right| d x+\sum_{i=1}^N \int_{R_{l_1,l_2}}\left|f_i(x)-\sum_{k=1}^{N_i} \lambda_{k, i} f_{k, i}(x)\right| d x .\\
		\end{aligned}
		$$
		Combining \eqref{BKBan_1}, \eqref{BKBan_2}, and \eqref{BKBan_3}, it gives that
		$$
		\begin{aligned}
			& \int_{R_{l_1,l_2}}\left|f(x)-\sum_{i=1}^N \sum_{k=1}^{N_i} \lambda_{k, i} f_{k, i}(x)\right| d x \\
			& \quad \lesssim\left\|\chi_{R_{l_1,l_2}}\right\|_{M\vec{\dot{K}}_{p^{\prime}, q^{\prime}}^{-\alpha, \lambda}}\left(\sum_{i=N+1}^{\infty}\left\|f_i\right\|_{{B \vec{\dot{K}}_{p, q}^{\alpha, \lambda}}}+\sum_{i=1}^N\left\|f_i-\sum_{k=1}^{N_i} \lambda_{k, i} f_{k, i}\right\|_{{B \vec{\dot{K}}_{p, q}^{\alpha, \lambda}}}\right) \\
			& \quad \lesssim\left\|\chi_{R_{l_1,l_2}}\right\|_{M\vec{\dot{K}}_{p^{\prime}, q^{\prime}}^{-\alpha, \lambda}}\left(\epsilon+\sum_{i=1}^N 2^{-i} \epsilon\right) \lesssim\left\|\chi_{R_{l_1,l_2}}\right\|_{M\vec{\dot{K}}_{p^{\prime}, q^{\prime}}^{-\alpha, \lambda}} \epsilon
		\end{aligned}
		$$
	As a consequence, $\sum_{i=1}^{\infty} \sum_{k=1}^{\infty} \lambda_{k, i} f_{k, i}$ converges to $f$ in $L_{ {loc }}^1\left(\mathbb{R}^n \times \R^m\right)$. Hence, $\sum_{i=1}^{\infty}$ $\sum_{k=1}^{\infty} \lambda_{k, i} f_{k, i}$ converges to $f$ locally in measure. Therefore, there exists a subsequence of $\left\{\sum_{i=1}^N \sum_{k=1}^M \lambda_{k, i} f_{k, i}\right\}_{N, M}$ that converges to $f$ almost everywhere. Furthermore, the coefficients $\lambda_{k,i}$, $i,k\in\mathbb N$, satisfy $\sum_{i=1}^{\infty} \sum_{k=1}^{\infty}\left|\lambda_{k, i}\right| \leq \sum_{i=1}^{\infty}(1+\epsilon)\left\|f_i\right\|_{B \vec{\dot{K}}_{p, q}^{\alpha, \lambda}}$.
		That is, $\sum_{i=1}^{\infty} f_i$ converges to $f$ in ${B \vec{\dot{K}}_{p, q}^{\alpha, \lambda}}$. Since $\epsilon$ is arbitrary, we obtain
		\begin{align}\label{Banach}
			\left\|\sum_{i=1}^{\infty} f_i\right\|_{B \vec{\dot{K}}_{p, q}^{\alpha, \lambda}} \leq \sum_{i=1}^{\infty}\left\|f_i\right\|_{B \vec{\dot{K}}_{p, q}^{\alpha, \lambda}}. 
		\end{align}
		Hence, it follows immediately from \eqref{Banach} and Proposition \ref{pre.extra.} that ${B \vec{\dot{K}}_{p, q}^{\alpha, \lambda}}$ is a Banach lattice.
	\end{proof}
	Our second main result concerns the boundedness of the strong maximal operator on the product block-Herz space ${B \vec{\dot{K}}_{p,q}^{\alpha,\lambda}}\left(\mathbb R^n\times\mathbb R^m\right)$. The following lemma, which is taken from \cite{Wei2021_1}, is crucial to the proof.
	\begin{lemma}\label{FM.}
		Let $1<p, q<\infty$, and
		$$
		\max \left\{-\frac{n}{p},-\frac{m}{p}\right\}<\alpha<\min \left\{n\left(1-\frac{1}{p}\right), m\left(1-\frac{1}{p}\right)\right\} .
		$$
		Then
		$$
		\|\chi_R\|_{\vec{\dot{K}}_{p,q}^\alpha\left(\mathbb{R}^n \times \mathbb{R}^m\right)}\|\chi_R\|_{\vec{\dot{K}}_{p',q'}^{-\alpha}\left(\mathbb{R}^n \times \mathbb{R}^m\right)} \approx |R| .
		$$
	\end{lemma}
	
	The above properties of product block-Herz space ${B\vec{\dot{K}}_{p, q}^{\alpha, \lambda}}\left( {{\R^n} \times {\R^m}} \right)$ and Lemma \ref{FM.} can be derived as follows.
	\begin{theorem}\label{boundedMS}
		Let $\alpha \in \mathbb{R}$, $1<p,q<\infty$, 
		$0<\lambda<\min\{-\alpha+n/p',-\alpha+m/p'\}$, and $$
		\max \left\{-\frac{n}{p},-\frac{m}{p}\right\}<\alpha<\min \left\{n\left(1-\frac{1}{p}\right), m\left(1-\frac{1}{p}\right)\right\}.
		$$ Then the strong maximal operator $M_S$ is bounded on ${B \vec{\dot{K}}_{p,q}^{\alpha, \lambda}}\left( {{\R^n} \times {\R^m}} \right)$.
	\end{theorem}
	\begin{proof}
		We will prove that for any $b \in b \vec{\dot{K}}_{p,q}^{\alpha, \lambda}$, $\|M_Sb\|_{B\vec{\dot{K}}_{p,q}^{\alpha, \lambda}} \lesssim 1$.
		Since $B\vec{\dot{K}}_{p,q}^{\alpha, \lambda} \subseteq L_{\mathrm{loc}}$ by Theorem \ref{BK.Ban}, the strong maximal operator $M_S$ is well defined on $B \vec{\dot{K}}_{p,q}^{\alpha, \lambda}$. Let $b \in {b \vec{\dot{K}}_{p, q}^{\alpha, \lambda}}$  supported in $R_{l_1,l_2}$. 
		Define
		\[
		E_{0,0}=R_{l_1+1,l_2+1},\quad
		E_{i,0}=I_{l_1+i+1}\times R_{l_2+1},\quad
		E_{0,j}=R_{l_1+1}\times J_{l_2+j+1},
		\]
		for $i,j\geq1$, set $E_{i,j}=I_{l_1+i+1}\times J_{l_2+j+1}$, and put $m_{i,j}=\chi_{E_{i,j}}M_Sb$ for $(i,j)\in\mathbb N_0^2$. Then $M_Sb=\sum_{i,j\in\mathbb N_0}m_{i,j}$ and $\operatorname{supp}m_{i,j}\subset R_{l_1+i+1,l_2+j+1}$. Lemma \ref{MSHerz} gives that $M_S$ is bounded on $\vec{\dot{K}}_{p,q}^\alpha$, which implies
		\begin{align*}
			\left\|m_{0,0}\right\|_{\vec{\dot{K}}_{p,q}^\alpha} \lesssim\left\|M_S b\right\|_{\vec{\dot{K}}_{p,q}^\alpha} \lesssim\|b\|_{\vec{\dot{K}}_{p,q}^\alpha} \lesssim 2^{-(l_1+l_2)\lambda }.
		\end{align*}
		Thus $m_{0,0}$ is a constant multiple of a product Herz-block, and hence $\|m_{0,0}\|_{B\vec{\dot{K}}_{p,q}^{\alpha,\lambda}}\lesssim1$.
		To estimate the scales of $M_S b$ on the remaining sets $E_{i,j}$, $(i,j)\in\mathbb N_0^2\setminus\{(0,0)\}$, we need to consider the following terms.
		By Hölder's inequality, we have
		$$
		\begin{aligned}
			m_{i, j}=\chi_{E_{i,j}} M_S b & \lesssim \frac{\chi_{E_{i,j}}}{2^{(i+l_1)n} \times 2^{(j+l_2)m}} \int_{R_{l_1,l_2}}|b(z)| d z \\
			& \lesssim  \frac{\chi_{E_{i,j}}}{2^{(i+l_1)n}\times 2^{(j+l_2)m}}\|b\|_{\vec{\dot{K}}_{p,q}^\alpha}\left\|\chi_{{R_{l_1,l_2}}}\right\|_{\vec{\dot{K}}_{p',q'}^{-\alpha}}
		\end{aligned}
		$$
		Taking the $\vec{\dot{K}}_{p,q}^\alpha$ norm in the above inequalities, and then combining this with the definition of a product Herz-block, we obtain
		\begin{align*}
			\left\|m_{i, j}\right\|_{\vec{\dot{K}}_{p,q}^\alpha}
			&\lesssim \frac{\left\|\chi_{E_{i,j}}\right\|_{\vec{\dot{K}}_{p,q}^\alpha}}{2^{(i+l_1)n} \times 2^{(j+l_2)m}}\|b\|_{\vec{\dot{K}}_{p,q}^\alpha}\left\|\chi_{R_{l_1,l_2}}\right\|_{\vec{\dot{K}}_{p',q'}^{-\alpha}}\\
			& \lesssim \frac{\left\|\chi_{E_{i,j}}\right\|_{\vec{\dot{K}}_{p,q}^\alpha}\left\|\chi_{R_{l_1,l_2}}\right\|_{\vec{\dot{K}}_{p',q'}^{-\alpha}}}{2^{(i+l_1)n} \times 2^{(j+l_2)m} \times 2^{(l_1+l_2)\lambda}}. \\
		\end{align*}
		Lemma \ref{FM.} gives
		\begin{align*}
			\left\|\chi_{E_{i,j}}\right\|_{\vec{\dot{K}}_{p,q}^\alpha}\left\|\chi_{R_{l_1,l_2}}\right\|_{\vec{\dot{K}}_{p',q'}^{-\alpha}} &\lesssim\left\|\chi_{R_{i+l_1+1, j+l_2+1}}\right\|_{\vec{\dot{K}}_{p,q}^\alpha}\left\|\chi_{R_{l_1,l_2}}\right\|_{\vec{\dot{K}}_{p',q'}^{-\alpha}}\\
			&\lesssim \frac{\left\|\chi_{R_{l_1,l_2}}\right\|_{\vec{\dot{K}}_{p',q'}^{-\alpha}}|{R_{i+l_1+1, j+l_2+1}}|}{\left\|\chi_{R_{i+l_1+1, j+l_2+1}}\right\|_{\vec{\dot{K}}_{p',q'}^{-\alpha}}}.
		\end{align*}
		Then we have 
		\begin{align*}
			\left\|m_{i, j}\right\|_{\vec{\dot{K}}_{p,q}^\alpha} \lesssim \frac{\left\|\chi_{R_{l_1,l_2}}\right\|_{\vec{\dot{K}}_{p',q'}^{-\alpha}}}{\left\|\chi_{R_{i+l_1+1, j+l_2+1}}\right\|_{\vec{\dot{K}}_{p',q'}^{-\alpha}}} 2^{-(l_1+l_2)\lambda} 
			\approx  \frac{\left\|\chi_{{R_{l_1,l_2}}}\right\|_{\vec{\dot{K}}_{p',q'}^{-\alpha}}}{\left\|\chi_{R_{i+l_1+1, j+l_2+1}}\right\|_{\vec{\dot{K}}_{p',q'}^{-\alpha}}} \frac{2^{(i+j+2)\lambda}}{2^{(i+j+l_1+l_2+2)\lambda}}.
		\end{align*}
	For any $(i,j)\in\mathbb N_0^2\setminus\{(0,0)\}$, choose a constant $C_0>0$, independent of $i,j,l_1,l_2$, large enough so that we define 
	\begin{equation*}
	\sigma_{i, j} := C_0 2^{(i+j+2)\lambda} \frac{\left\|\chi_{R_{l_1,l_2}}\right\|_{\vec{\dot{K}}_{p',q'}^{-\alpha}}}{\left\|\chi_{R_{i+l_1+1, j+l_2+1}}\right\|_{\vec{\dot{K}}_{p',q'}^{-\alpha}}} \ \text{and} \ b_{i,j} := \left(\sigma_{i, j}\right)^{-1} m_{i ,j}.
	\end{equation*}
		
		By \eqref{ineq.1},  
		$$
		\sum\limits_{\substack{i,j \in \mathbb N_0\\ i+j\ge1}} {{\sigma _{ij}}}  
		\approx 
		\sum\limits_{\substack{i,j \in \mathbb N_0\\ i+j\ge1}}
		2^{i(\lambda+\alpha-n/p')}
		2^{j(\lambda+\alpha-m/p')}
		< \infty.
		$$
		
		It is easy to see 
		\begin{align*}
	\left\| b_{i,j} \right\|_{\vec{\dot{K}}_{p,q}^\alpha} \le 2^{-(i+j+l_1+l_2+2)\lambda} \quad \text{and} \quad
			M_S b \le m_{0,0}+ \sum_{\substack{i,j \in \mathbb N_0\\ i+j\ge1}} \sigma_{i, j} b_{i,j}.
		\end{align*}
		By the definition of ${B \vec{\dot{K}}_{p,q}^{\alpha, \lambda}}$, $m_{0,0}+\sum_{\substack{i,j \in \mathbb N_0\\ i+j\ge1}} \sigma_{i, j} b_{i,j} \in {B \vec{\dot{K}}_{p,q}^{\alpha, \lambda}}$.
		By Proposition \ref{pre.extra.} and Theorem \ref{BK.Ban}, we get
		\begin{align*}
			\left\|M_S b\right\|_{B \vec{\dot{K}}_{p,q}^{\alpha, \lambda}} \leq \left\| m_{0,0}+\sum_{\substack{i,j \in \mathbb N_0\\ i+j\ge1}} \sigma_{i, j} b_{i,j}\right\|_{B \vec{\dot{K}}_{p,q}^{\alpha, \lambda}} \lesssim  1+\sum_{\substack{i,j \in \mathbb N_0\\ i+j\ge1}} \sigma_{i, j} < \infty.
		\end{align*}
		
		For any $f\in {B \vec{\dot{K}}_{p,q}^{\alpha, \lambda}}$, we can write $f = \sum\limits_{k = 1}^\infty  {{\lambda _k}{b_k}}$, where ${\left\| {\left\{ {{\lambda _k}} \right\}_{k = 1}^\infty } \right\|_{{l^1}}} \le 2\left\| f \right\|_{B \vec{\dot{K}}_{p,q}^{\alpha, \lambda}}.$
		Hence, Proposition \ref{pre.extra.} and Theorem \ref{BK.Ban} give
		$$\left\| {{M_S}f} \right\|_{B \vec{\dot{K}}_{p,q}^{\alpha, \lambda}} \le \left\| {\sum\limits_{k = 1}^\infty  {\left| {{\lambda _k}} \right|{M_S}{b_k}} } \right\|_{B \vec{\dot{K}}_{p,q}^{\alpha, \lambda}} \le\sum\limits_{k = 1}^\infty  {\left| {{\lambda _k}} \right|\left\| {{M_S}{b_k}} \right\|_{B \vec{\dot{K}}_{p,q}^{\alpha, \lambda}}} 
		\lesssim_{n,p,q,\alpha,\lambda} \left\| f \right\|_{B \vec{\dot{K}}_{p,q}^{\alpha, \lambda}}.$$
		The proof is done.
	\end{proof}
	\section{Extrapolation on product Morrey-Herz spaces}\label{sec4}
	In the preceding section, we demonstrated the boundedness of the strong maximal operator within product block-Herz spaces, highlighting the crucial role of the duality between product Morrey-Herz spaces and weighted product block-Herz spaces. In this section, we then study extrapolation theories applicable to product Morrey-Herz spaces.
	Specifically, Theorem \ref{boundedMS} establishes that the strong maximal operator $M_S$ exhibits boundedness on the space ${B \vec{\dot{K}}_{(p/p_0)^{\prime}, (q/p_0)^{\prime}}^{-p_0\alpha, p_0\lambda}}\left( {{\R^n} \times {\R^m}} \right)$.
	We formally define the operator $\mathfrak{R}$ as
	$$
	\mathfrak{R} h=\sum_{k=0}^{\infty} \frac{M_S^k(|h|)}{2^k\left\|M_S\right\|_{{B \vec{\dot{K}}_{(p/p_0)^{\prime}, (q/p_0)^{\prime}}^{-p_0\alpha, p_0\lambda}\left(\mathbb{R}^n \times \mathbb{R}^m\right)}}^k}
	$$
	where $h \in L_{{loc }}\left(\mathbb{R}^n \times \mathbb{R}^m\right), M_S^0$ is the identity operator, $M_S^k$ is the $k$-th iterate of $M_S$ for $k \in \mathbb{N} \backslash\{0\}$ and $\left\|M_S\right\|_{B \vec{\dot{K}}_{(p/p_0)^{\prime}, (q/p_0)^{\prime}}^{-p_0\alpha, p_0\lambda}(\mathbb{R}^n \times \mathbb{R}^m)}$ denotes the operator norm of $M_S$ on the space ${B \vec{\dot{K}}_{(p/p_0)^{\prime}, (q/p_0)^{\prime}}^{-p_0\alpha, p_0\lambda}}\left( {{\R^n} \times {\R^m}} \right)$ as long as $M_S$ is bounded on ${B \vec{\dot{K}}_{(p/p_0)^{\prime}, (q/p_0)^{\prime}}^{-p_0\alpha, p_0\lambda}}\left( {{\R^n} \times {\R^m}} \right)$. The following proposition ensures that $\mathfrak{R}$ is well-defined.
	\begin{proposition}
		Let $\alpha \in \mathbb{R}$, $0<p_0<p,q<\infty$, $0<\lambda<\min\{\alpha+n/p,\alpha+m/p\}$, and 
		$$\max \left\{-\frac{n}{p},-\frac{m}{p}\right\}<\alpha<\min \left\{n\left(\frac{1}{p_0}-\frac{1}{p}\right), m\left(\frac{1}{p_0}-\frac{1}{p}\right)\right\}.$$
		Then the operator $\mathfrak{R}$ is well-defined on ${B \vec{\dot{K}}_{(p/p_0)^{\prime}, (q/p_0)^{\prime}}^{-p_0\alpha, p_0\lambda}}\left( {{\R^n} \times {\R^m}} \right)$ and satisfies
		\begin{align}\label{R_1}
			& |h| \leq \mathfrak{R} h,\\ \label{R_2}
			& \|\mathfrak{R} h\|_{B \vec{\dot{K}}_{(p/p_0)^{\prime}, (q/p_0)^{\prime}}^{-p_0\alpha, p_0\lambda}\left(\mathbb{R}^n \times \mathbb{R}^m\right)}  \leq 2 \|h\|_{B \vec{\dot{K}}_{(p/p_0)^{\prime}, (q/p_0)^{\prime}}^{-p_0\alpha, p_0\lambda}\left(\mathbb{R}^n \times \mathbb{R}^m\right)}, \\ \label{R_3}
			& [\mathfrak{R} h]_{A_1^*} \leq 2\|M_S\|_{B \vec{\dot{K}}_{(p/p_0)^{\prime}, (q/p_0)^{\prime}}^{-p_0\alpha, p_0\lambda}\left(\mathbb{R}^n \times \mathbb{R}^m\right)}.
		\end{align}
	\end{proposition}
	\begin{proof}
		It is easy to see that $\mathfrak{R}$ is well-defined on ${B \vec{\dot{K}}_{(p/p_0)^{\prime}, (q/p_0)^{\prime}}^{-p_0\alpha, p_0\lambda}}$ by Theorem \ref{boundedMS}.
		The inequality \eqref{R_1} is immediate. Notice that
		$$
		\begin{aligned}
			M_S(\mathfrak{R} h) & \leq \sum_{k=0}^{\infty} \frac{M_S^{k+1}(|h|)}{2^k\left\|M_S\right\|_{B \vec{\dot{K}}_{(p/p_0)^{\prime}, (q/p_0)^{\prime}}^{-p_0\alpha, p_0\lambda}}^k}
			\\&=
			2\left\|M_S\right\|_{B \vec{\dot{K}}_{(p/p_0)^{\prime}, (q/p_0)^{\prime}}^{-p_0\alpha, p_0\lambda}} \sum_{k=1}^{\infty} \frac{M_S^k(|h|)}{2^k\left\|M_S\right\|_{B \vec{\dot{K}}_{(p/p_0)^{\prime}, (q/p_0)^{\prime}}^{-p_0\alpha, p_0\lambda}}^k} \\
			& \leq 2\left\|M_S\right\|_{B \vec{\dot{K}}_{(p/p_0)^{\prime}, (q/p_0)^{\prime}}^{-p_0\alpha, p_0\lambda}} \mathfrak{R} h.
		\end{aligned}
		$$
		\eqref{R_2} and \eqref{R_3} are consequences of the boundedness of $M_S$ on ${B \vec{\dot{K}}_{(p/p_0)^{\prime}, (q/p_0)^{\prime}}^{-p_0\alpha, p_0\lambda}}$.
	\end{proof}
	We are now ready to state the extrapolation theorem on product Morrey-Herz spaces.
	\begin{theorem}\label{extra.ori.}
		Let $\alpha \in \mathbb{R}$, $0<p_0<p,q<\infty$, $0<\lambda<\min\{\alpha+n/p,\alpha+m/p\}$, and $$\max \left\{-\frac{n}{p},-\frac{m}{p}\right\}<\alpha<\min \left\{n\left(\frac{1}{p_0}-\frac{1}{p}\right), m\left(\frac{1}{p_0}-\frac{1}{p}\right)\right\}.$$
		Suppose that for every
		$$
		v\in\left\{\mathfrak{R} h: \|h\|_{{B \vec{\dot{K}}_{(p/p_0)^{\prime}, (q/p_0)^{\prime}}^{-p_0\alpha, p_0\lambda}}\left(\mathbb{R}^n \times \mathbb{R}^m\right)} \leq 1\right\},
		$$
		and every $(f, g) \in \mathfrak{F}$, 
		\begin{align}\label{extra.ori_1}
			\int_{\mathbb{R}^n \times \R^m} |f(x)|^{p_0} v(x) d x \lesssim \int_{\mathbb{R}^n \times \R^m} |g(x)|^{p_0} v(x) d x<\infty.
		\end{align}
		Then for every $(f, g) \in \mathfrak{F}$, we have
		\begin{align}\label{extra.ori_2}
			\|f\|_{{M \vec{\dot{K}}_{p, q}^{\alpha, \lambda}}\left(\mathbb{R}^n \times \mathbb{R}^m\right)} \lesssim\|g\|_{{M \vec{\dot{K}}_{p, q}^{\alpha, \lambda}}\left(\mathbb{R}^n \times \mathbb{R}^m\right)}
		\end{align}
		Moreover, for $\left(f_i, g_i\right) \in \mathfrak{F}$, $i \in \mathbb{N}$, satisfying \eqref{extra.ori_1}, there holds
		\begin{align}\label{extra.ori_3}
			\left\|\left(\sum_{i \in \mathbb{N}}\left|f_i\right|^{p_0}\right)^{\frac{1}{p_0}}\right\|_{{M \vec{\dot{K}}_{p, q}^{\alpha, \lambda}}\left(\mathbb{R}^n \times \mathbb{R}^m\right)} \lesssim\left\|\left(\sum_{i \in \mathbb{N}}\left|g_i\right|^{p_0}\right)^{\frac{1}{p_0}}\right\|_{{M \vec{\dot{K}}_{p, q}^{\alpha, \lambda}}\left(\mathbb{R}^n \times \mathbb{R}^m\right)}
		\end{align}
	\end{theorem}
	\begin{proof}
		Our proof uses the Rubio de Francia iteration algorithm; see \cite{Cruz2011}. It suffices to prove \eqref{extra.ori_2}, since the proof of \eqref{extra.ori_3} is similar to that of \cite[Corollary 3.12]{Cruz2011}.
		Let $(f, g) \in \mathfrak{F}$ with $g \in {M \vec{\dot{K}}_{p, q}^{\alpha,\lambda}}$. For any non-negative $h \in {B \vec{\dot{K}}_{(p/p_0)^{\prime}, (q/p_0)^{\prime}}^{-p_0\alpha, p_0\lambda}}$ with $\|h\|_{B \vec{\dot{K}}_{(p/p_0)^{\prime}, (q/p_0)^{\prime}}^{-p_0\alpha, p_0\lambda}} \leq 1$, since $g \in {M \vec{\dot{K}}_{p, q}^{\alpha,\lambda}}$ and $|h| \leq \mathfrak{R} h$, we get from \eqref{extra.ori_1} with $v=\mathfrak{R} h$ and by \eqref{R_2},
		\begin{align*}
			\int_{\mathbb{R}^n \times \mathbb{R}^m} |f(z)|^{p_0} h(z) d z 
			&\lesssim \int_{\mathbb{R}^n \times \mathbb{R}^m} |f(z)|^{p_0} \mathfrak{R} h(z) d z
			\lesssim \int_{\mathbb{R}^n \times \mathbb{R}^m} |g(z)|^{p_0} \mathfrak{R} h(z) d z \\
			& \lesssim \||g|^{p_0}\|_{M \vec{\dot{K}}_{p/p_0, q/p_0}^{p_0\alpha, p_0\lambda}}
			\|\mathfrak{R} h\|_{B \vec{\dot{K}}_{(p/p_0)^{\prime}, (q/p_0)^{\prime}}^{-p_0\alpha, p_0\lambda}} \\
			&\lesssim \|g\|^{p_0}_{M \vec{\dot{K}}_{p, q}^{\alpha, \lambda}}\| h\|_{B \vec{\dot{K}}_{(p/p_0)^{\prime}, (q/p_0)^{\prime}}^{-p_0\alpha, p_0\lambda}}.
		\end{align*}
		It is easy to see that 
		\begin{align*}
			\sup \left\{\int_{\mathbb{R}^n \times \mathbb{R}^m} |f(z)|^{p_0} h(z) d z : 0\le h,\ \|h\|_{{B \vec{\dot{K}}_{(p/p_0)^{\prime}, (q/p_0)^{\prime}}^{-p_0\alpha, p_0\lambda}}} \leq 1\right\} \lesssim \|g\|^{p_0}_{M \vec{\dot{K}}_{p, q}^{\alpha, \lambda}}
		\end{align*} 
		Hence, we get
		\begin{align*}
			\|f\|^{p_0}_{M \vec{\dot{K}}_{p, q}^{\alpha, \lambda}}
			& =\||f|^{p_0}\|_{M \vec{\dot{K}}_{p/p_0, q/p_0}^{p_0\alpha, p_0\lambda}}= \sup \left\{\int_{\mathbb{R}^n \times \mathbb{R}^m} |f(z)|^{p_0} h(z) d z : \|h\|_{{B \vec{\dot{K}}_{(p/p_0)^{\prime}, (q/p_0)^{\prime}}^{-p_0\alpha, p_0\lambda}}} \leq 1\right\} \\
			& \lesssim \|g\|^{p_0}_{M \vec{\dot{K}}_{p, q}^{\alpha, \lambda}}.
		\end{align*}
		The proof is accomplished.
	\end{proof}
	The boundedness of certain operators on product Morrey-Herz spaces can be deduced from Theorem~\ref{extra.ori.} as follows.
	\begin{theorem}\label{extra.T.}
		Let $\alpha \in \mathbb{R}$, $0<p_0<p,q<\infty$, $0<\lambda<\min\{\alpha+n/p,\alpha+m/p\}$, and $$\max \left\{-\frac{n}{p},-\frac{m}{p}\right\}<\alpha<\min \left\{n\left(\frac{1}{p_0}-\frac{1}{p}\right), m\left(\frac{1}{p_0}-\frac{1}{p}\right)\right\}.$$
		Suppose that for every
		$h \in {B \vec{\dot{K}}_{(p/p_0)^{\prime}, (q/p_0)^{\prime}}^{-p_0\alpha, p_0\lambda}\left(\mathbb{R}^n \times \mathbb{R}^m\right)}$ with $\|h\|_{B \vec{\dot{K}}_{(p/p_0)^{\prime}, (q/p_0)^{\prime}}^{-p_0\alpha, p_0\lambda}\left(\mathbb{R}^n \times \mathbb{R}^m\right)} \leq 1$,
		$T$ is bounded on $L^{p_0}\left(\mathbb{R}^n \times \mathbb{R}^m,\mathfrak{R}h\right)$, with a bound independent of $h$. Then $T$ can be extended to a bounded operator on ${{M \vec{\dot{K}}_{p, q}^{\alpha, \lambda}}\left(\mathbb{R}^n \times \mathbb{R}^m\right)}$.
		Moreover, we have
		\begin{align}\label{extra.T_3}
			\left\|\left(\sum_{i \in \mathbb{N}}\left|Tf_i\right|^{p_0}\right)^{\frac{1}{p_0}}\right\|_{{M \vec{\dot{K}}_{p, q}^{\alpha, \lambda}}\left(\mathbb{R}^n \times \mathbb{R}^m\right)} \lesssim\left\|\left(\sum_{i \in \mathbb{N}}\left|f_i\right|^{p_0}\right)^{\frac{1}{p_0}}\right\|_{{M \vec{\dot{K}}_{p, q}^{\alpha, \lambda}}\left(\mathbb{R}^n \times \mathbb{R}^m\right)}.
		\end{align}
	\end{theorem}
	\begin{proof}
		For any $f \in {M \vec{\dot{K}}_{p, q}^{\alpha, \lambda}}$, $h \in {B \vec{\dot{K}}_{(p/p_0)^{\prime}, (q/p_0)^{\prime}}^{-p_0\alpha, p_0\lambda}}$ with $\|h\|_{B \vec{\dot{K}}_{(p/p_0)^{\prime}, (q/p_0)^{\prime}}^{-p_0\alpha, p_0\lambda}} \leq 1$, by Hölder's inequality and \eqref{R_2}, 
		\begin{align*}
			\int_{\mathbb{R}^n \times \mathbb{R}^m} |f(z)|^{p_0} \mathfrak{R} h(z) d z
			&\lesssim \||f|^{p_0}\|_{M \vec{\dot{K}}_{p/p_0, q/p_0}^{p_0\alpha, p_0\lambda}}\|\mathfrak{R} h\|_{B \vec{\dot{K}}_{(p/p_0)^{\prime}, (q/p_0)^{\prime}}^{-p_0\alpha, p_0\lambda}} \\
			& \lesssim \|f\|^{p_0}_{M \vec{\dot{K}}_{p, q}^{\alpha, \lambda}}\| h\|_{B \vec{\dot{K}}_{(p/p_0)^{\prime}, (q/p_0)^{\prime}}^{-p_0\alpha, p_0\lambda}} < \infty.\label{extra.T_4}
		\end{align*}
		Set $\mathfrak{F}=\left\{(|T f|,|f|): f \in {M \vec{\dot{K}}_{p, q}^{\alpha, \lambda}}\right\}$, and then the boundedness of $T$ assures that \eqref{extra.ori_1} is valid for the pairs $\mathfrak{F}$. The desired results can follow from Theorem \ref{extra.ori.}.
	\end{proof}
	\section{Applications}\label{sec5}
	This section gives several applications of the extrapolation theorem on product Morrey-Herz spaces. We first establish a bi-parameter version of the Fefferman-Stein vector-valued inequality. We then prove a John-Nirenberg inequality and obtain a new characterization of little bmo in terms of the product Morrey-Herz norm. Finally, we apply the extrapolation theorem to the boundedness of bi-parameter operators and their commutators on product Morrey-Herz spaces.
	\subsection{Fefferman-Stein vector-valued strong maximal inequalities on product Morrey-Herz spaces}~~
	
	The classical one-parameter Fefferman-Stein vector-valued maximal inequalities have been extended to various Morrey-type spaces, including Morrey spaces \cite{Sawano2007,Tang2005}, Morrey spaces with variable exponents \cite{Ho2015}, and Morrey spaces associated with Banach function spaces \cite{Ho2017_2}. Consequently, Theorem \ref{F.S.ineq.} extends the classical Fefferman-Stein vector-valued maximal inequalities to the bi-parameter setting of product Morrey-Herz spaces.
	
	For every \( p_0 > 1 \) and any \( v \in A_1^*(\mathbb{R}^n \times \mathbb{R}^m) \subseteq A_{p_0}^*(\mathbb{R}^n \times \mathbb{R}^m) \), \( M_S \) is bounded on \( L^{p_0}(v) \). According to Theorem \ref{extra.T.}, the Fefferman-Stein vector-valued strong maximal inequalities on product Morrey-Herz spaces are as follows.
	\begin{theorem}\label{F.S.ineq.}
		Let $\alpha \in \mathbb{R}$, $1<{p_0}<p,q<\infty$, $0<\lambda<\min\{\alpha+n/p,\alpha+m/p\}$, and $$\max \left\{-\frac{n}{p},-\frac{m}{p}\right\}<\alpha<\min \left\{n\left(\frac{1}{p_0}-\frac{1}{p}\right), m\left(\frac{1}{p_0}-\frac{1}{p}\right)\right\}.$$
		Then $M_S$ is bounded on ${M \vec{\dot{K}}_{p, q}^{\alpha, \lambda}}\left(\mathbb{R}^n \times \mathbb{R}^m\right)$. Moreover, we have
		\begin{equation}\label{F.S.ineq_2}
			\left\|\left(\sum_{i \in \mathbb{N}}\left|M_S f_i\right|^{p_0}\right)^{\frac{1}{p_0}}\right\|_{M \vec{\dot{K}}_{p, q}^{\alpha, \lambda}\left(\mathbb{R}^n \times \mathbb{R}^m\right)} \lesssim\left\|\left(\sum_{i \in \mathbb{N}}\left|f_i\right|^{p_0}\right)^{\frac{1}{p_0}}\right\|_{M \vec{\dot{K}}_{p, q}^{\alpha, \lambda}\left(\mathbb{R}^n \times \mathbb{R}^m\right)}.
		\end{equation}
	\end{theorem}
	
	\subsection{John-Nirenberg inequality and a new characterization of bmo in terms of product Morrey-Herz spaces}~~
	
	First, we recall the definition of little bmo. A function \(f\in L_{\mathrm{loc}}(\mathbb R^n\times\mathbb R^m)\) belongs to \(\operatorname{bmo}(\mathbb R^n\times\mathbb R^m)\) if
	$$
	\|f\|_{\text{bmo}(\mathbb{R}^n \times \mathbb{R}^m)} = \sup_{R \in \mathcal{R}} \frac{\left\| (f - f_R) \chi_R \right\|_{L^1(\mathbb{R}^n \times \mathbb{R}^m)}}{\left\| \chi_R \right\|_{L^1(\mathbb{R}^n \times \mathbb{R}^m)}}<\infty,
	$$
	where \( f_R = \frac{1}{|R|} \int_R f(z) \, dz \).
	An important result is that \( f \in \operatorname{bmo}(\mathbb{R}^n \times \mathbb{R}^m) \) is equivalent to \( f(x, \cdot) \in \operatorname{BMO}(\mathbb{R}^m) \) uniformly in \( x \in \mathbb{R}^n \) and, symmetrically, \( f(\cdot, y) \in \operatorname{BMO}(\mathbb{R}^n) \) uniformly in \( y \in \mathbb{R}^m \), cf. \cite{Holmes2018}.
	
	In 2019, Hart and Torres \cite{Hart2019} established the John-Nirenberg inequality for little BMO, which states that for any \( \gamma > 0 \) and \( R \in \mathcal{R} \),
	\begin{align}\label{bmo_1}
		\left| \left\{ z \in R : \left| f(z) - f_R \right| > \gamma \right\} \right| \leq C_1 |R| e^{-\frac{C_2 \gamma}{\|f\|_{\text{bmo}}}}, \quad f \in \operatorname{bmo}(\mathbb{R}^n \times \mathbb{R}^m) \setminus \mathcal{C},
	\end{align}
	where \(\mathcal C\) is the set of constant functions and \(C_1,C_2>0\) are independent of \(f\), \(R\), and \(\gamma\). The John-Nirenberg inequality provides equivalent characterizations of little bmo; see \cite{Hart2019, Holmes2018} for further details. Below we extend this inequality to product Morrey-Herz spaces.
	
	First, we need to estimate the product Morrey-Herz norm of \( \chi_R \), where \( R \in \mathcal{R} \).
	\begin{lemma}\label{lembmo}
		Let $\alpha \in \mathbb{R}$, $1<p,q<\infty$, $0<\lambda<\min\{\alpha+n/p,\alpha+m/p\}$, and $$
		\max \left\{-\frac{n}{p},-\frac{m}{p}\right\}<\alpha<\min \left\{n\left(1-\frac{1}{p}\right), m\left(1-\frac{1}{p}\right)\right\}.
		$$ Then for any $R \in \mathcal{R}$,
		$$
		\left\|\chi_R\right\|_{M \vec{\dot{K}}_{p, q}^{\alpha, \lambda}\left(\mathbb{R}^n \times \mathbb{R}^m\right)}\left\|\chi_R\right\|_{B \vec{\dot{K}}_{p^{\prime}, q^{\prime}}^{-\alpha, \lambda}\left(\mathbb{R}^n \times \mathbb{R}^m\right)} \approx |R|.
		$$
	\end{lemma}
	\begin{proof}
		Theorem \ref{PHMdual} gives the inequality $|R| \lesssim\left\|\chi_R\right\|_{M \vec{\dot{K}}_{p, q}^{\alpha, \lambda}}\left\|\chi_R\right\|_{{B \vec{\dot{K}}_{p^{\prime}, q^{\prime}}^{-\alpha, \lambda}}}$, we only need to prove the opposite.
		For any $R \in \mathcal{R}$ and $g \in L_{ {loc }}\left(\mathbb{R}^n \times \mathbb{R}^m\right)$, we define
		$$
		P_R\left(g\right) (z)=\left(\frac{1}{|R|} \int_R\left|g\left(z^{\prime}\right)\right| d z^{\prime}\right) \chi_R(z) .
		$$
		It is easy to see $P_R g \lesssim M_S g,$ and 
		$
		\sup _{R \in \mathcal{R}}\left\|P_R\right\|_{{B \vec{\dot{K}}_{p^{\prime}, q^{\prime}}^{-\alpha, \lambda}}} \lesssim\left\|M_S\right\|_{{B \vec{\dot{K}}_{p^{\prime}, q^{\prime}}^{-\alpha, \lambda}}}.
		$
		
		By Theorem \ref{boundedMS},
		$$
		\begin{aligned}
			& \left\|\chi_R\right\|_{{M \vec{\dot{K}}_{p, q}^{\alpha, \lambda}}}\left\|\chi_R\right\|_{{B \vec{\dot{K}}_{p^{\prime}, q^{\prime}}^{-\alpha, \lambda}}
			} \\
			= & \sup \left\{\left|\int_R g(z) d z\right|\left\|\chi_R\right\|_{{B \vec{\dot{K}}_{p^{\prime}, q^{\prime}}^{-\alpha, \lambda}}
			}:\|g\|_{{B \vec{\dot{K}}_{p^{\prime}, q^{\prime}}^{-\alpha, \lambda}}
			} \leq 1\right\} \\
			\leq & \sup \left\{|R|\left\|P_R g\right\|_{{B \vec{\dot{K}}_{p^{\prime}, q^{\prime}}^{-\alpha, \lambda}}
			}: \|g\|_{{B \vec{\dot{K}}_{p^{\prime}, q^{\prime}}^{-\alpha, \lambda}}
			} \leq 1\right\} \\
			\lesssim & |R| .
		\end{aligned}
		$$
		This proves the assertion.
	\end{proof}
	\begin{theorem}\label{_end}
		Let $\alpha \in \mathbb{R}$, $1<p,q<\infty$, $0<\lambda<\min\{\alpha+n/p,\alpha+m/p\}$, and $$
	\max \left\{-\frac{n}{p},-\frac{m}{p}\right\}<\alpha<\min \left\{n\left(1-\frac{1}{p}\right), m\left(1-\frac{1}{p}\right)\right\}.$$ Then for every $f\in\operatorname{bmo}(\mathbb R^n\times\mathbb R^m)\setminus\mathcal C$ there exist constants $C_1,C_2>0$ such that, for any $\gamma>0$ and $R \in \mathcal{R}$,
		$$
		\left\|\chi_{\left\{z \in R:\left|f(z)-f_R\right|>\gamma\right\}}\right\|_{{M \vec{\dot{K}}_{p, q}^{\alpha, \lambda}}\left(\mathbb{R}^n \times \mathbb{R}^m\right)} \leq C_1 e^{-\frac{C_2 \gamma}{\|f\|_{\mathrm{bmo}}}}\left\|\chi_R\right\|_{{M \vec{\dot{K}}_{p, q}^{\alpha, \lambda}}\left(\mathbb{R}^n \times \mathbb{R}^m\right)} .
		$$
	\end{theorem}
	\begin{proof}
	By the weighted John-Nirenberg inequality for $A_\infty^*$ weights and by \eqref{R_3}, there exist constants $D,E>0$, independent of the Rubio weight $v$, such that, for every $R\in\mathcal R$ and $\gamma>0$,
		$$
		v\left(\left\{z \in R:\left|f(z)-f_R\right|>\gamma\right\}\right) \leq D e^{-\frac{E\gamma}{\|f\|_{\mathrm{bmo}}}} v(R),
		$$
	which is equivalent to
		$$
		\int_{\mathbb{R}^n \times \mathbb{R}^m} \chi_{\left\{z \in R:\left|f(z)-f_R\right|>\gamma\right\}}(z) v(z) d z \leq D e^{-\frac{E \gamma}{\|f\|_{\mathrm{bmo}}}} \int_R v(z) d z .
		$$
		Set $p_0=1$ and
		$
		\mathfrak{F}=\left\{\left(\chi_{\left\{z \in R:\left|f(z)-f_R\right|>\gamma\right\}}, e^{-\frac{E \gamma}{\|f\|_{\mathrm{bmo}}}} \chi_R\right): R \in \mathcal{R}\right\}.
		$
		By Theorem \ref{extra.ori.},
		$$
		\left\|\chi_{\left\{z \in R:\left|f(z)-f_R\right|>\gamma\right\}}\right\|_{{M \vec{\dot{K}}_{p, q}^{\alpha, \lambda}}} \leq C_1 e^{-\frac{C_2 \gamma}{\|f\|_{\mathrm{bmo}}}}\left\|\chi_R\right\|_{{M \vec{\dot{K}}_{p, q}^{\alpha, \lambda}}} .
		$$
	\end{proof}
	The aforementioned result can be utilized to establish a new characterization of bmo on product Morrey-Herz spaces. Hence, we aim to demonstrate that the norm
	$$
	\|f\|_{\text{bmo}_{M \vec{\dot{K}}_{p, q}^{\alpha, \lambda}}\left(\mathbb{R}^n \times \mathbb{R}^m\right)} = \sup_{R \in \mathcal{R}} \frac{\left\| (f - f_R) \chi_R \right\|_{M \vec{\dot{K}}_{p, q}^{\alpha, \lambda}\left(\mathbb{R}^n \times \mathbb{R}^m\right)}}{\left\| \chi_R \right\|_{M \vec{\dot{K}}_{p, q}^{\alpha, \lambda}\left(\mathbb{R}^n \times \mathbb{R}^m\right)}}
	$$
	is equivalent to \( \|f\|_{\text{bmo}\left(\mathbb{R}^n \times \mathbb{R}^m\right)} \) for \( f \in \operatorname{bmo}(\mathbb{R}^n \times \mathbb{R}^m) \).
	We now characterize little bmo in terms of product Morrey-Herz spaces; for related results, see \cite{Ho2017_1, Ho2018_1}.
	\begin{theorem}
		Let $\alpha \in \mathbb{R}$, $1<p,q<\infty$, $0<\lambda<\min\{\alpha+n/p,\alpha+m/p\}$, and $$
		\max \left\{-\frac{n}{p},-\frac{m}{p}\right\}<\alpha<\min \left\{n\left(1-\frac{1}{p}\right), m\left(1-\frac{1}{p}\right)\right\}.
	$$ Then, on $\operatorname{bmo}(\mathbb R^n\times\mathbb R^m)$ modulo constants, the norms $\|\cdot\|_{\mathrm{bmo}_{M \vec{\dot{K}}_{p, q}^{\alpha, \lambda}}\left(\mathbb{R}^n \times \mathbb{R}^m\right)}$ and $\|\cdot\|_{\mathrm{bmo}\left(\mathbb{R}^n \times \mathbb{R}^m\right)}$ are equivalent.
	\end{theorem}
	\begin{proof}
		On the one hand, by Lemma \ref{chara.PMK}, $\chi_R \in {M \vec{\dot{K}}_{p, q}^{\alpha, \lambda}}$. Theorem \ref{PHMdual} implies that
		$$
		\int_R\left|f(z)-f_R\right| d z \leq\left\|\left(f(z)-f_R\right) \chi_R\right\|_{{M \vec{\dot{K}}_{p, q}^{\alpha, \lambda}}}\left\|\chi_R\right\|_{{B \vec{\dot{K}}_{p^{\prime}, q^{\prime}}^{-\alpha, \lambda}}} .
		$$
		By the characteristic-function estimate following from Lemma \ref{FM.} and the definitions of the product Morrey-Herz and product block-Herz norms, we get
		$$
		\int_R\left|f(z)-f_R\right| d z \lesssim|R| \frac{\left\|\left(f(z)-f_R\right) \chi_R\right\|_{{M \vec{\dot{K}}_{p, q}^{\alpha, \lambda}}}}{\left\|\chi_R\right\|_{{M \vec{\dot{K}}_{p, q}^{\alpha, \lambda}}}} .
		$$
		Then
		\begin{align}\label{nbmo_1}
			\|f\|_{\mathrm{bmo}} \lesssim\|f\|_{\mathrm{bmo}_{{M \vec{\dot{K}}_{p, q}^{\alpha, \lambda}}}} .
		\end{align}
		On the other hand, for any $j \in \mathbb{Z}$, the John-Nirenberg inequality on ${M \vec{\dot{K}}_{p, q}^{\alpha, \lambda}}$ gives
		\begin{align}\label{nbmo_2}
			\left\|\chi_{\left\{z \in R: 2^j<\left|f(z)-f_R\right| \leq 2^{j+1}\right\}}\right\|_{{M \vec{\dot{K}}_{p, q}^{\alpha, \lambda}}} \leq C_1 e^{-\frac{C_2 2^j}{\|f\|_{\mathrm{bmo}}}}\left\|\chi_R\right\|_{{M \vec{\dot{K}}_{p, q}^{\alpha, \lambda}}} .
		\end{align}
		For any $f \in \operatorname{bmo}\left(\mathbb{R}^n \times \mathbb{R}^m\right)$, there exists some $k \in \mathbb{Z}$ satisfying
		\begin{align}\label{nbmo_3}
			2^k<\|f\|_{\mathrm{bmo}} \leq 2^{k+1} .
		\end{align}
		Therefore, for all $R \in \mathcal{R}$,
		\begin{align}\label{nbmo_4}
			\left\|\left(f(z)-f_R\right) \chi_R\right\|_{{M \vec{\dot{K}}_{p, q}^{\alpha, \lambda}}} \leq 2^k\left\|\chi_R\right\|_{{M \vec{\dot{K}}_{p, q}^{\alpha, \lambda}}}+\sum_{j=k}^{\infty} 2^{j+1}\left\|\chi_{\left\{z \in R: 2^j<\left|f(z)-f_R\right| \leq 2^{j+1}\right\}}\right\|_{{M \vec{\dot{K}}_{p, q}^{\alpha, \lambda}}} .
		\end{align}
		Combining \eqref{nbmo_2} and \eqref{nbmo_3}, 
		$$
		\left\|\left(f(z)-f_R\right) \chi_R\right\|_{{M \vec{\dot{K}}_{p, q}^{\alpha, \lambda}}} \leq 2^k\left\|\chi_R\right\|_{{M \vec{\dot{K}}_{p, q}^{\alpha, \lambda}}}+C \sum_{j=k}^{\infty} 2^{j+1} e^{-\frac{C_2 2^j}{\|f\|_{\mathrm{bmo}}}}\left\|\chi_R\right\|_{{M \vec{\dot{K}}_{p, q}^{\alpha, \lambda}}} .
		$$
		A simple computation gives
		$$
		\sum_{j=k}^{\infty} 2^{j+1} e^{-\frac{C_2 2^j}{\|f\|_{\mathrm{bmo}}}} \lesssim \int_0^{\infty} \exp \left(-\frac{C_2 s}{\|f\|_{\mathrm{bmo}}}\right) d s \lesssim\|f\|_{\mathrm{bmo}}.
		$$
		By \eqref{nbmo_3}, we have
		\begin{align}\label{nbmo_5}
			\frac{\left\|\left(f(z)-f_R\right) \chi_R\right\|_{{M \vec{\dot{K}}_{p, q}^{\alpha, \lambda}}}}{\left\|\chi_R\right\|_{{M \vec{\dot{K}}_{p, q}^{\alpha, \lambda}}}} \lesssim 2^k+\|f\|_{\mathrm{bmo}} \lesssim\|f\|_{\mathrm{bmo}} .
		\end{align}
		The proof is completed from \eqref{nbmo_1} and \eqref{nbmo_5}.
	\end{proof}
	\subsection{Bi-parameter Calderón-Zygmund operators and their commutators on product Morrey-Herz spaces}~~

	Extrapolation theory in product Lebesgue spaces, as explored in \cite{Ho2016_1, Ho2017_3, Koki.2018}, has been shown to help establish boundedness for multiparameter operators. This subsection aims to employ Theorem \ref{extra.T.} to extend this theory to encompass bi-parameter operators and their commutators within the framework of product Morrey-Herz spaces.
	
	We next recall the bi-parameter Calderón-Zygmund operator introduced by Fefferman and Stein in \cite{Fe.1982}. The operator $T_K$ is defined by $T_K f(z)=(f*K)(z)$, $z\in\mathbb R^n\times\mathbb R^m$, where the kernel $K$ satisfies the following conditions:
	\begin{itemize}
		\item[(a)] For each $y \in \mathbb{R}^m, 0<\alpha<\beta$,
		$$
		\int_{\alpha<|x|<\beta} K(x, y) d x=0 ;
		$$
		and for each $x \in \mathbb{R}^n, 0<\alpha<\beta$,
		$$
		\int_{\alpha<|y|<\beta} K(x, y) d y=0 .
		$$
		\item[(b)] For $x \in \mathbb{R}^n \backslash\{0\}$ and $y \in \mathbb{R}^m \backslash\{0\}$,
		$$
		|K(x, y)| \lesssim \frac{1}{|x|^n|y|^m}
		$$
		\item[(c)] There exists an $\eta>0$ such that, for any $h \in \mathbb{R}^n$ with $|x|>2|h|$ and $y \in \mathbb{R}^m \backslash\{0\}$,
		$$
		|K(x+h, y)-K(x, y)| \lesssim \frac{1}{|y|^m}\left(\frac{|h|}{|x|}\right)^\eta \frac{1}{|x|^n},
		$$
		and for any $k \in \mathbb{R}^m$ with $|y|>2|k|$ and $x \in \mathbb{R}^n \backslash\{0\}$,
		$$
		|K(x, y+k)-K(x, y)| \lesssim \frac{1}{|x|^n}\left(\frac{|k|}{|y|}\right)^\eta \frac{1}{|y|^m}
		$$
		\item[(d)] For any $h \in \mathbb{R}^n, k \in \mathbb{R}^m$ with $|x|>2|h|$ and $|y|>2|k|$, write $\Delta_h^1 K(x, y)=K(x+h, y)-K(x, y)$ and $\Delta_k^2 K(x, y)=$ $K(x, y+k)-K(x, y), K$ satisfies
		$$
		\left|\Delta_k^2 \Delta_h^1 K(x, y)\right| \lesssim \frac{1}{|x|^n|y|^m}\left(\frac{|h|}{|x|} \frac{|k|}{|y|}\right)^\eta,
		$$
		where $\eta>0$ is the same as in condition (c).
	\end{itemize}
	
	By \cite[Theorem 3]{Fe.1982}, if $1<p_0<\infty$ and $v \in A_1^*\left(\mathbb{R}^n \times \mathbb{R}^m\right) \subseteq A_{p_0}^*\left(\mathbb{R}^n \times \mathbb{R}^m\right)$, then $T_K$ is bounded on $L^{p_0}(v)$ with constants depending only on the relevant weight characteristics. Therefore, Theorem~\ref{extra.T.} yields the following result.
	\begin{theorem}\label{end_1}
		Let $\alpha \in \mathbb{R}$, $1<{p_0}<p,q<\infty$, $0<\lambda<\min\{\alpha+n/p,\alpha+m/p\}$, and $$\max \left\{-\frac{n}{p},-\frac{m}{p}\right\}<\alpha<\min \left\{n\left(\frac{1}{p_0}-\frac{1}{p}\right), m\left(\frac{1}{p_0}-\frac{1}{p}\right)\right\}.$$ 
		Then $T_K$ can be extended to a bounded operator on ${M \vec{\dot{K}}_{p, q}^{\alpha, \lambda}}\left( {{\R^n} \times {\R^m}} \right)$. Furthermore, we have
		$$
		\left\|\left(\sum_{i \in \mathbb{N}}\left|T_K f_i\right|^{p_0}\right)^{\frac{1}{p_0}}\right\|_{M \vec{\dot{K}}_{p, q}^{\alpha, \lambda}\left(\mathbb{R}^n \times \mathbb{R}^m\right)} \lesssim\left\|\left(\sum_{i \in \mathbb{N}}\left|f_i\right|^{p_0}\right)^{\frac{1}{p_0}}\right\|_{M \vec{\dot{K}}_{p, q}^{\alpha, \lambda}\left(\mathbb{R}^n \times \mathbb{R}^m\right)}
		$$
	\end{theorem}
	The mapping property of the commutator for the bi-parameter singular integral operator $T_K$ provides another application. The seminal work of Coifman, Rochberg, and Weiss \cite{Coifman1976} showed that commutators of Calderón-Zygmund singular integral operators are bounded on Lebesgue spaces when the symbol $b$ belongs to $\operatorname{BMO}\left(\mathbb{R}^N\right)$. Since then, numerous characterizations of $\operatorname{BMO}\left(\mathbb{R}^N\right)$ have been obtained through the boundedness of commutators on various function spaces, including Morrey-type spaces \cite{Tao2020,Wang2020}, variable spaces \cite{Tan2017,Wang2019}, and Lorentz spaces \cite{Dao2021}.
	
	In recent work, Duong et al. characterized \(\operatorname{BMO}\) through the boundedness of commutators associated with certain bi-parameter integral operators \cite{Duong2021,Duong2018}. Building on this idea, we obtain a characterization of little \(\operatorname{bmo}\) via the boundedness of the commutator of the bi-parameter singular integral operator \(T_K\) on product Morrey-Herz spaces.
	
	Recall that the commutator \([b,T_K]\), associated with a bi-parameter Calderón-Zygmund singular integral operator \(T_K\) and a locally integrable function \(b\), is defined by \([b,T_K](f)(z)=b(z)T_K(f)(z)-T_K(bf)(z)\). According to \cite[Theorem 3.17]{Benyi2020}, if \(b\in \operatorname{bmo}(\mathbb R^n\times\mathbb R^m)\), then \([b,T_K]\) is bounded on \(L^{p_0}(\mathbb R^n\times\mathbb R^m,v)\) for each \(1<p_0<\infty\) and every weight \(v\in A_1^*(\mathbb R^n\times\mathbb R^m)\subseteq A_{p_0}^*(\mathbb R^n\times\mathbb R^m)\). We obtain the following two results.
	\begin{theorem}
		Let $\alpha \in \mathbb{R}$, $1<p_0<p,q<\infty$, $0<\lambda<\min\{\alpha+n/p,\alpha+m/p\}$, and $$\max \left\{-\frac{n}{p},-\frac{m}{p}\right\}<\alpha<\min \left\{n\left(\frac{1}{p_0}-\frac{1}{p}\right), m\left(\frac{1}{p_0}-\frac{1}{p}\right)\right\}.$$
		If $b \in \operatorname{bmo}\left(\mathbb{R}^n \times \mathbb{R}^m\right)$, then $\left[b, T_K\right]$ is bounded on ${M \vec{\dot{K}}_{p, q}^{\alpha, \lambda}}\left( {{\R^n} \times {\R^m}} \right)$.
	\end{theorem}
	
	\begin{theorem}
		Let $\alpha \in \mathbb{R}$, $1<p,q<\infty$, $0<\lambda<\min\{\alpha+n/p,\alpha+m/p\}$, and $$
		\max \left\{-\frac{n}{p},-\frac{m}{p}\right\}<\alpha<\min \left\{n\left(1-\frac{1}{p}\right), m\left(1-\frac{1}{p}\right)\right\}.$$
	Suppose that $b\in L_{\mathrm{loc}}(\mathbb R^n\times\mathbb R^m)$ and $K(x,y)=K_1(x)K_2(y)$, where $K_1$ and $K_2$ are homogeneous of degrees $-n$ and $-m$, respectively, and where $1/K_1$ and $1/K_2$ admit absolutely convergent Fourier series on some balls $B_1\subseteq\mathbb R^n$ and $B_2\subseteq\mathbb R^m$, respectively.
		If $\left[b, T_K\right]$ is a bounded operator on ${M \vec{\dot{K}}_{p, q}^{\alpha, \lambda}}\left(\mathbb{R}^n \times \mathbb{R}^m\right)$, then $b \in \operatorname{bmo}\left(\mathbb{R}^n \times \mathbb{R}^m\right)$.
	\end{theorem}
	\begin{proof}
		The proof is similar to those of Janson \cite{Janson1987} and Paluszyński \cite{Palu1995}.
		Let $B\left(a_0, \delta_1 \sqrt{n}\right)$ be the ball centered at $a_0$ with radius $\delta_1 \sqrt{n}$, such that $1 / K_1(x)$ can be represented as an absolutely convergent Fourier series
		$$
		\frac{1}{K_1(x)}=\sum c_k e^{i \mu_k \cdot x}
		$$
		where $c_k\in\mathbb{C}$ and $\mu_k\in\mathbb{R}^n$ are given constants and $\sum\left|c_k\right|<\infty$.
		
		Set $x_1=a_0 / \delta_1$. Notice that for all $x$ such that $\left|x-x_1\right|<\sqrt{n}$, we have
		$$
		\frac{1}{K_1(x)}=\frac{\delta_1^{-n}}{K_1\left(\delta_1 x\right)}=\delta_1^{-n} \sum c_k e^{i \delta_1 \mu_k \cdot x} .
		$$
		For given cubes $Q_1=Q_1\left(x_0, r\right)$ and $Q_1^{\prime}=Q_1\left(x_0-r x_1, r\right)$, if $x \in Q_1$ and $x^{\prime} \in Q_1^{\prime}$, we have
		$$
		\left|\frac{x-x^{\prime}}{r}-x_1\right| \leq\left|\frac{x-x_0}{r}\right|+\left|\frac{x^{\prime}-\left(x_0-r x_1\right)}{r}\right|<\sqrt{n} .
		$$
		For $K_2(y)$, the argument is similar to that for $K_1(x)$. Let $B\left(b_0, \delta_2 \sqrt{m}\right)$ be the ball centered at $b_0$ with radius $\delta_2 \sqrt{m}$, such that $1 / K_2(y)$ can be represented as an absolutely convergent Fourier series
		$$
		\frac{1}{K_2(y)}=\sum d_l e^{i v_l \cdot y}
		$$
		where $d_l\in\mathbb{C}$ and $v_l\in\mathbb{R}^m$ are given constants and $\sum\left|d_l\right|<\infty$.
		Set $y_1=b_0 / \delta_2$. Notice that for all $y$ such that $\left|y-y_1\right|<\sqrt{m}$, we have
		$$
		\frac{1}{K_2(y)}=\frac{\delta_2^{-m}}{K_2\left(\delta_2 y\right)}=\delta_2^{-m} \sum d_l e^{i \delta_2 v_l \cdot y}
		$$
		For given cubes $Q_2=Q_2\left(y_0, s\right)$ and $Q_2^{\prime}=Q_2\left(y_0-s y_1, s\right)$, if $y \in Q_2$ and $y^{\prime} \in Q_2^{\prime}$, then
		$$
		\left|\frac{y-y^{\prime}}{s}-y_1\right| \leq\left|\frac{y-y_0}{s}\right|+\left|\frac{y^{\prime}-\left(y_0-s y_1\right)}{s}\right|<\sqrt{m} .
		$$
		For any $R=Q_1\left(x_0, r\right) \times Q_2\left(y_0, s\right) \in \mathcal{R}$, let $R^{\prime}=Q_1^{\prime} \times Q_2^{\prime}$ and $t(x, y)=\operatorname{sgn}(b(x, y)-$ $\left.b_{R^{\prime}}\right)$. Then we have the following estimates:
		$$
		\begin{aligned} & \frac{1}{|R|} \int_R\left|b(x, y)-b_R\right| d x d y \\ \lesssim & \frac{1}{|R|} \int_R\left|b(x, y)-b_{R^{\prime}}\right| d x d y \\ = & \frac{1}{|R|} \int_R\left(b(x, y)-b_{R^{\prime}}\right) t(x, y) d x d y \\ = & \frac{1}{|R|} \frac{1}{\left|R^{\prime}\right|} \int_R \int_{R^{\prime}}\left(b(x, y)-b\left(x^{\prime}, y^{\prime}\right)\right) d x^{\prime} d y^{\prime} t(x, y) d x d y \\ = & \frac{1}{|R|^2} \int_R \int_{R^{\prime}} t(x, y)\left(b(x, y)-b\left(x^{\prime}, y^{\prime}\right)\right) \frac{r^n K_1\left(x-x^{\prime}\right)}{K_1\left(\frac{x-x^{\prime}}{r}\right)} \frac{s^m K_2\left(y-y^{\prime}\right)}{K_2\left(\frac{y-y^{\prime}}{s}\right)} d x^{\prime} d y^{\prime} d x d y \\ \approx & \frac{1}{|R|} \int_{\mathbb{R}^n \times \mathbb{R}^m} \int_{\mathbb{R}^n \times \mathbb{R}^m} t(x, y)\left(b(x, y)-b\left(x^{\prime}, y^{\prime}\right)\right) K\left(x-x^{\prime}, y-y^{\prime}\right) \\ \times & \sum_k c_k e^{i \delta_1 \mu_k \cdot \frac{x-x^{\prime}}{r}} \sum_l d_l e^{i \delta_2 v_l \cdot \frac{y-y^{\prime}}{s}} \chi_{Q_1}(x) \chi_{Q_2}(y) \chi_{Q_1^{\prime}}\left(x^{\prime}\right) \chi_{Q_2^{\prime}}\left(y^{\prime}\right) d x^{\prime} d y^{\prime} d x d y .\end{aligned}
		$$
		Setting
		$$
		\begin{aligned}
			& g_{k, l}(x, y)=e^{i \frac{\delta_1}{r} \mu_k \cdot x} e^{i \frac{\delta_2}{s} v_l \cdot y} t(x, y) \chi_{Q_1}(x) \chi_{Q_2}(y), \\
			& h_{k, l}\left(x^{\prime}, y^{\prime}\right)=e^{-i \frac{\delta_1}{r} \mu_k \cdot x^{\prime}} e^{-i \frac{\delta_2}{s} v_l \cdot y^{\prime}} \chi_{Q_1^{\prime}}\left(x^{\prime}\right) \chi_{Q_2^{\prime}}\left(y^{\prime}\right),
		\end{aligned}
		$$
		one can further obtain
		$$
		\frac{1}{|R|} \int_R\left|b(x, y)-b_R\right| d x d y
		$$
		$$
		\begin{aligned}
			& \approx \frac{1}{|R|} \sum_k \sum_l c_k d_l \int_{\mathbb{R}^n \times \mathbb{R}^m} \int_{\mathbb{R}^n \times \mathbb{R}^m}\left(b(x, y)-b\left(x^{\prime}, y^{\prime}\right)\right) \\
			& \times K\left(x-x^{\prime}, y-y^{\prime}\right) g_{k, l}(x, y) h_{k, l}\left(x^{\prime}, y^{\prime}\right) d x^{\prime} d y^{\prime} d x d y \\
			& =\frac{1}{|R|} \sum_k \sum_l c_k d_l \int_{\mathbb{R}^n \times \mathbb{R}^m}\left[b, T_K\right]\left(h_{k, l}\right)(x, y) g_{k, l}(x, y) d x d y .
		\end{aligned}
		$$
		Using the boundedness of $\left[b, T_K\right]$ on ${M \vec{\dot{K}}_{p, q}^{\alpha, \lambda}}\left(\mathbb{R}^n \times \mathbb{R}^m\right)$ and Theorem \ref{PHMdual}, we obtain
		$$
		\begin{aligned}
			& \frac{1}{|R|} \int_R\left|b(x, y)-b_R\right| d x d y \\
			\lesssim & \frac{1}{|R|} \sum_k \sum_l\left|c_k\right|\left|d_l\right|\left\|\left[b, T_K\right]\left(h_{k, l}\right)\right\|_{{M \vec{\dot{K}}_{p, q}^{\alpha, \lambda}}\left(\mathbb{R}^n \times \mathbb{R}^m\right)}\left\|g_{k, l}\right\|_{{B \vec{\dot{K}}_{p^{\prime}, q^{\prime}}^{-\alpha, \lambda}}} \\
			\lesssim & \frac{1}{|R|} \sum_k \sum_l\left|c_k\right|\left|d_l\right|\left\|h_{k, l}\right\|_{{M \vec{\dot{K}}_{p, q}^{\alpha, \lambda}}}\left\|g_{k, l}\right\|_{{B \vec{\dot{K}}_{p^{\prime}, q^{\prime}}^{-\alpha, \lambda}}} \\
			\lesssim & \frac{1}{|R|} \sum_k \sum_l\left|c_k\right|\left|d_l\right|\left\|\chi_R\right\|_{{B \vec{\dot{K}}_{p^{\prime}, q^{\prime}}^{-\alpha, \lambda}}}\left\|\chi_{R^{\prime}}\right\|_{{M \vec{\dot{K}}_{p, q}^{\alpha, \lambda}}} .
		\end{aligned}
		$$
		From the choice of $R^{\prime}$, we can find a proper rectangle $R_0 \subseteq \mathbb{R}^n \times \mathbb{R}^m$ such that $R, R^{\prime} \subseteq R_0$ and $\left|R_0\right| \approx|R|$. The lattice properties of the product Morrey-Herz and product block-Herz norms give
		\begin{equation}\label{end_2}
			\frac{1}{|R|} \int_R\left|b(x, y)-b_R\right| d x d y \lesssim \frac{1}{|R|} \sum_k \sum_l\left|c_k\right|\left|d_l\right|\left\|\chi_{R_0}\right\|_{{B \vec{\dot{K}}_{p^{\prime}, q^{\prime}}^{-\alpha, \lambda}}}\left\|\chi_{R_0}\right\|_{{M \vec{\dot{K}}_{p, q}^{\alpha, \lambda}}}.
		\end{equation}
		Using the characteristic-function estimate following from Lemma \ref{FM.} and the definitions of the product Morrey-Herz and product block-Herz norms, together with \eqref{end_2}, we arrive at
		$$
		\frac{1}{|R|} \int_R\left|b(x, y)-b_R\right| d x d y \lesssim \frac{\left|R_0\right|}{|R|} \sum_k \sum_l\left|c_k\right|\left|d_l\right| \lesssim 1 .
		$$
		This proves the assertion.
	\end{proof}

	\begin{Remark}
	We remark that the boundedness of bi-parameter singular integral operators on the classical product Morrey space $M_{u,p}\left(\mathbb R^n\times\mathbb R^m\right)$ was established in \cite{Wei2023_2}, while the corresponding boundedness result on the product Herz space ${\vec{\dot K}_{p,q}^{\alpha}}\left(\mathbb R^n\times\mathbb R^m\right)$ was obtained in \cite{Wei2021_1}. To the best of our knowledge, the Fefferman-Stein vector-valued strong maximal inequality, the John-Nirenberg inequality, the characterization of little bmo, and the boundedness of bi-parameter Calderón-Zygmund operators and their commutators on the product Morrey-Herz space ${M\vec{\dot K}_{p,q}^{\alpha,\lambda}}\left(\mathbb R^n\times\mathbb R^m\right)$ are new.
	\end{Remark}
	
	\vspace{1cm}

	\noindent{\bf Acknowledgments } 
Cen would like to express his deepest gratitude to Dr. Zhicheng Tong (Jilin University) for his unwavering support over the years, which has been instrumental in helping him overcome numerous challenges. He also thanks his advisor, Prof. Xinfeng Wu, for invaluable guidance and steadfast mentorship. Finally, the authors acknowledge the anonymous editors and reviewers for their careful reading of the manuscript and their insightful comments, which improved both the content and the presentation of this work.

	\medskip 
	
	\noindent{\bf Data Availability} Our manuscript has no associated data.

	\medskip 
	\noindent{\bf\Large Declarations}
	\medskip 
	
	\noindent{\bf Conflict of interest} The authors declare that they have no conflict of interest.

\end{document}